\documentclass[11pt]{amsart}
\usepackage{amscd,amsmath,amsthm,amssymb}
\usepackage{verbatim}
\usepackage[all]{xy}
\usepackage{color}
\usepackage{hyperref}
\usepackage{tikz, float} \usetikzlibrary {positioning}

\usepackage[lined,algonl,boxed,norelsize]{algorithm2e}
\let\chapter\undefined

\def\NZQ{\mathbb}               % the font for N,Z,Q,R,C

\def\ZZ{{\NZQ Z}}
\def\RR{{\NZQ R}}

\def\EE{{\mathbb E}}
\def\frk{\mathfrak}               % font for "Fraktur"
\def\aa{{\frk a}}

\def\mm{{\frk m}}

\def\Phi{{\frk n}}
\def\Phi{{\frk N}}

\def\Sf{{\frk S}}

\def\Tf{{\mathfrak T}}

\def\MM{{\mathbf M}}

\def\OO{{\mathbf O}}

\def\Rb{{\mathbf R}}
\def\Sb{{\mathbf S}}

\def\vb{{\mathbf v}}

\def\xb{{\mathbf x}}
\def\yb{{\mathbf y}}
\def\zb{{\mathbf z}}
\def\pb{{\mathbf p}}
\def\cb{{\mathbf c}}
\def\mb{{\mathbf m}}

\def\sb{{\mathbf s}}
\def\rb{{\mathbf r}}

\def\B{{\mathcal B}}
\def\Oc{{\mathcal O}}

\def\A{{\mathcal A}}

\def\MB{{\mathcal B}}

\def\M{{\mathcal M}}
\def\H{{\mathcal H}}

\def\Tc{{\mathcal T}}

\def\Sc{{\mathcal S}}
\def\P{{\mathcal P}}

\def\C{{\mathcal C}}

\def\js{{\mathsf j}}

\def\MB{{\mathcal B}}

\def\Sc{{\mathcal S}}
\def\M{{\mathcal M}}

\def\xb{{\mathbf x}}
\def\yb{{\mathbf y}}
\def\zb{{\mathbf z}}
\def\js{{\mathsf j}}

\def\opn#1#2{\def#1{\operatorname{#2}}} % to make operators
\opn\chara{char} \opn\length{\ell} \opn\pd{pd} \opn\rk{rk}
\opn\projdim{proj\,dim} \opn\injdim{inj\,dim} \opn\rank{rank}
\opn\depth{depth} \opn\grade{grade} \opn\height{height}
\opn\embdim{emb\,dim} \opn\codim{codim}
\def\OO{{\mathcal O}}
\opn\Tr{Tr} \opn\bigrank{big\,rank}
\opn\superheight{superheight}\opn\lcm{lcm}
\opn\trdeg{tr\,deg}%\emph{
\opn\reg{reg} \opn\lreg{lreg} \opn\ini{in} \opn\lpd{lpd}
\opn\size{size} \opn\sdepth{sdepth}
\opn\link{link}\opn\fdepth{fdepth}\opn\lex{lex}
\opn\LM{LM}
\opn\LC{LC}
\opn\NF{NF}
\opn\Merge{Merge}
\opn\sgn{sgn}
\opn\div{div} \opn\Div{Div} \opn\cl{cl} \opn\Pic{Pic}
\opn\Prin{Prin}
\opn\op{op}
\opn\indeg{indeg} \opn\outdeg{outdeg}
\opn\red{red}
\opn\Spec{Spec} \opn\Supp{Supp} \opn\supp{supp} \opn\Sing{Sing}
\opn\Ass{Ass} \opn\Min{Min}\opn\Mon{Mon}
\opn\Ann{Ann} \opn\Rad{Rad} \opn\Soc{Soc}
 \opn\Ker{Ker} \opn\Coker{Coker} \opn\Am{Am}
\opn\Hom{Hom} \opn\Tor{Tor} \opn\Ext{Ext} \opn\End{End}
\opn\Aut{Aut} \opn\id{id}

\opn\nat{nat}
\opn\pff{pf}%   \pf exists already
\opn\Pf{Pf} \opn\GL{GL} \opn\SL{SL} \opn\mod{mod} \opn\ord{ord}
\opn\Gin{Gin} \opn\Hilb{Hilb}\opn\sort{sort}
\opn\span{span}
\opn\Image{Image}
\opn\aff{aff} \opn\con{conv} \opn\relint{relint} \opn\st{st}
\opn\lk{lk} \opn\cn{cn} \opn\core{core} \opn\vol{vol}
\opn\link{link} \opn\star{star}\opn\lex{lex}\opn\set{set}
\opn\dist{dist}
\opn\gr{gr}

\def\pot#1#2{#1[\kern-0.28ex[#2]\kern-0.28ex]}

\opn\dirlim{\underrightarrow{\lim}}
\opn\inivlim{\underleftarrow{\lim}}
\let\tensor=\otimes

\let\to=\rightarrow

\def\Implies{\ifmmode\Longrightarrow \else
        \unskip${}\Longrightarrow{}$\ignorespaces\fi}
\def\implies{\ifmmode\Rightarrow \else
        \unskip${}\Rightarrow{}$\ignorespaces\fi}
\def\iff{\ifmmode\Longleftrightarrow \else
        \unskip${}\Longleftrightarrow{}$\ignorespaces\fi}

\let\:=\colon
\newtheorem{Theorem}{Theorem}[section]
\newtheorem{Lemma}[Theorem]{Lemma}
\newtheorem{Corollary}[Theorem]{Corollary}
\newtheorem{Proposition}[Theorem]{Proposition}

\theoremstyle{remark}
\newtheorem{Remark}[Theorem]{Remark}

\theoremstyle{definition}
\newtheorem{Example}[Theorem]{Example}

\newtheorem{Definition}[Theorem]{Definition}
\newtheorem*{Notation}{Notation}
\let\kappa=\varkappa
\def\qed{\ifhmode\textqed\fi
      \ifmmode\ifinner\quad\qedsymbol\else\dispqed\fi\fi}
\def\textqed{\unskip\nobreak\penalty50
       \hskip2em\hbox{}\nobreak\hfil\qedsymbol
       \parfillskip=0pt \finalhyphendemerits=0}
\def\dispqed{\rlap{\qquad\qedsymbol}}

\opn\dis{dis}
\def\pnt{{\raise0.5mm\hbox{\large\bf.}}}

\opn\Lex{Lex}
\opn\syz{{\rm syz}}
\opn\spoly{{\rm spoly}}
\opn\LM{{\rm LM}}
\opn\lm{{\rm lm}}
\opn\lcm{{\rm lcm}} \opn\A{\mathcal A}
\numberwithin{equation}{section}

\tikzstyle{Cwhite}=[scale = .8,circle, fill = white, minimum size=3mm] 
\tikzstyle{Cgray}=[draw=black, scale = .3,circle, fill = white, minimum size=3mm]
\tikzstyle{Cblack2}=[scale = .4,circle, fill = black, minimum size=3mm] 
\tikzstyle{Cblack}=[scale = .7,circle, fill = black, minimum size=3mm]
\tikzstyle{C0}=[scale = .9,circle, fill = black!0, inner sep = 0pt, minimum size=3mm]
\tikzstyle{C1}=[scale = .7,circle, fill = black!0, inner sep = 0pt, minimum size=3mm]
\tikzstyle{Cred}=[scale = .4,circle, fill = red, minimum size=3mm] 
\tikzstyle{Cblue}=[scale = .4,circle, fill =blue, minimum size=3mm] 
\begin{document}

\title{divisors on graphs, orientations, syzygies, \protect\\ and system reliability}

\subjclass[2010]{05C40, 13D02, 05E40, 13A02, 13P10}

\author {Fatemeh Mohammadi}

\address{Institut f\"ur Mathematik, Technische Universit\"at Berlin, 10623 Berlin, Germany }
\email{\tt fatemeh.mohammadi@math.tu-berlin.de}
%\date{21 April}

\keywords{System reliability, Betti numbers, polyhedral cellular complex,  orientations.}%spanning trees,

%%%%%%%%%%%%%%%%%%%%%%%%%%%%%%%%%%%%%%%%%%%%%%%%%%%%%%%%%%%%%%%%%%%%%%%%%%%%%%%%%%%%%%%%%%%%%%%%%%%%%%%%%%%%%%%
%%%%%%%%%%%%%%%%%%%%%%%%%%%%%%%%%%%%%%%%%%%%%%%%%%%%%%%%%%%%%%%%%%%%%%%%%%%%%%%%%%%%%%%%%%%%%%%%%%%%%%%%%%%%%%%
\begin{abstract}
We study various ideals arising in the theory of system reliability. We use ideas from  the theory of divisors, orientations and matroids on graphs to describe the minimal polyhedral cellular  resolutions of these ideals. In each case we give an explicit combinatorial description of the minimal generating set for each higher syzygy module
in terms of the acyclic orientations of the graph, 
the reduced divisors and the bounded regions of the graphic hyperplane arrangement. 
The resolutions of all these ideals are closely related, and their Betti numbers are independent of the characteristic of the base field. We apply these results to compute the reliability of their associated systems.
\end{abstract}

%%%%%%%%%%%%%%%%%%%%%%%%%%%%%%%%%%%%%%%%%%%%%%%%%%%%%%%%%%%%%%%%%%%%%%%%%%%%%%%%%%%%%%%%%%%%%%%%%%%%%%%%%%%%%%%
%%%%%%%%%%%%%%%%%%%%%%%%%%%%%%%%%%%%%%%%%%%%%%%%%%%%%%%%%%%%%%%%%%%%%%%%%%%%%%%%%%%%%%%%%%%%%%%%%%%%%%%%%%%%%%%

\maketitle

%\setcounter{tocdepth}{1}
%\tableofcontents

%%%%%%%%%%%%%%%%%%%%%%%%%%%%%%%%%%%%%%%%%%%%%%%%%%%%%%%%%%%%%%%%%%%%%%%%%%%%%%%%%%%%%%%%%%%%%%%%%%%%%%%%%%%%%%%
%%%%%%%%%%%%%%%%%%%%%%%%%%%%%%%%%%%%%%%%%%%%%%%%%%%%%%%%%%%%%%%%%%%%%%%%%%%%%%%%%%%%%%%%%%%%%%%%%%%%%%%%%%%%%%%

%%%%%%%%%%%%%%%%%%%%%%%%%%%%%%%%%%%%%%%%%%%%%%%%%%%%%%%%%%%%%%%%%%%%%%%%%%%%%%%%%%%%%%%%%%%%%%%%%%%%%%%%%%%%%%%

\section{Introduction}\label{sec:intro}

%%%%%%%%%%%%%%%%%%%%%%%%%%%%%%%%%%%%%%%%%%%%%%%%%%%%%%%%%%%%%%%%%%%%%%%%%%%%%%%%%%%%%%%%%%%%%%%%%%%%%%%%%%%%%%%
%%%%%%%%%%%%%%%%%%%%%%%%%%%%%%%%%%%%%%%%%%%%%%%%%%%%%%%%%%%%%%%%%%%%%%%%%%%%%%%%%%%%%%%%%%%%%%%%%%%%%%%%%%%%%%%

%%%%%%%%%%%%%%%%%%%%%%%%%%%%%%%%%%%%%%%%%%%%%%%%%%%%%%%%%%%%%%%%%%%%%%%%%%%%%%%%%%%%%%%%%%%%%%%%%%%%%%%%%%%%%%%
%%%%%%%%%%%%%%%%%%%%%%%%%%%%%%%%%%%%%%%%%%%%%%%%%%%%%%%%%%%%%%%%%%%%%%%%%%%%%%%%%%%%%%%%%%%%%%%%%%%%%%%%%%%%%%

This work is concerned with the development of new connections between the theory of oriented matroids, the theory of divisors on graphs,  and the theory of system reliability. 
Inspired by the work of %Diaconis and Sturmfels \cite{diaconis} applying algebraic techniques in probability and statistics, and following the %work by 
Naiman-Wynn \cite{naiman} and Giglio-Wynn \cite{giglio2004monomial} 
connecting  system reliability to  Hilbert functions of their associated ideals, we study reliability of  %(directed) 
networks through the lens of algebraic statistics and  geometric combinatorics.  
Our main contribution is to apply the {\em syzygy tool} from computational algebra to distinguish the (non-cancelling) terms in the reliability formula for various systems. This gives a more clear insight into the structure of each such system.

\medskip

The starting point of this paper is to study the following network flow reliability problem. Let $G=(V,E)$ be a graph. 
Assume that the vertices are reliable but each edge may fail (with the probability $1-p_e$).  A popular game in system reliability theory is to  compute the probability of the union of certain events under various restrictions. 
The classical method to compute the system reliability is to apply the inclusion-exclusion principle of probability theory which is computationally expensive. On the other hand, the system reliability formula is equal to the numerator of  Hilbert series of  a certain ideal associated to the network.  The special networks have been studied in \cite{giglio2004monomial}, and the general case was stated as an open problem. We recommend \cite[Sec.\,6]{Dohmen} and  the survey articles \cite{agrawal1984survey} by Agrawal-Barlow,  and \cite{johnson1988survey}  by  Johnson-Malek for an overview of the subject.

%\medskip
\vspace{-1mm}
\subsection{Source-to-terminal (ST) system}\label{subsec:ST}
 %A well-known example in reliability theory is the {\em ST reliability}. 
In this setting we denote the system by  $(G,s,t)$ when we fix two vertices of graph  $s$ (source)  and $t$ (target), and we study the probability that there exists at least one (oriented) path from $s$ to $t$. 
%We consider that vertices are reliable but edges may fail (with the probability $1-p_e$). 
The network fails to communicate $s$ and $t$ whenever there is a set of failing (removed) edges such that there is no path connecting $s$ and $t$ using only the remaining edges. Such a set of edges is called a {\em cut} in this context. On the contrary, a {\em path} is a set of working edges that connect $s$ and $t$. We say that the network is working whenever there is a path of working edges between $s$ and $t$. In the algebraic approach we associate a variable $x_e$ to each edge $e$ of $G$. We consider the polynomial ring $R=k[x_e:e\in E]$ over a field $k$. To a set of edges we associate the product of their corresponding variables. The ideal generated by the monomials associated to the minimal paths (respectively minimal cuts) between $s$ and $t$ is called a path ideal $\P_{s,t}$ (respectively the cut ideal $\C_{s,t}$). The evaluation of the numerator of the Hilbert series of either the cut ideal or the path ideal of $G$ using the probabilities of failure or function of each edge, gives us the ST reliability of the network. 

\vspace{-1mm}
\subsubsection{Polyhedral cellular free resolutions}
%Let $R = k[x_1, x_2, \dots, x_m]$ be the polynomial ring over a field $k$ on $m$ variables with its usual $\mathbb{Z}^m$-grading, and 
Let $I \subset R$ be an ideal generated by monomials 
$I = \langle m_1, m_2, \dots, m_\ell \rangle$.
A graded free resolution of $I$ is an exact sequence of the form 
\vspace{-2mm}
\[
\mathcal{F} \, : 0 \rightarrow \cdots \rightarrow F_{i} \xrightarrow{\varphi_{i}} F_{i-1} \rightarrow \cdots \rightarrow F_0 \xrightarrow{\varphi_{0}} I \rightarrow 0
\]
where all $F_i$'s are free $R$-modules and all differential maps $\varphi_i$'s are graded. The resolution is called {\em minimal free resolution} (MFR) if $\varphi_{i+1}(F_{i+1}) \subseteq \mm F_i$ for all $i \geq 0$, where $\mm=\langle x_e:\ e\in E\rangle$.  The $i$-th {\em Betti number} $\beta_{i}(I)$ of $I$ is the rank of $F_i$. The $i$-th {\em graded Betti number} in degree $\js \in \mathbb{Z}^m$, 
denoted by $\beta_{i,\js}(I)$, is the rank of the degree $\js$ part of $F_i$. These integers encode very subtle numerical information  about the ideal (e.g. its Hilbert series). For a system ideal $I$, we express the numerator of its multigraded Hilbert function as
\vspace{-1mm}
\[
1-\sum_{i=1}^d (-1)^{i+1}(\sum_{\js \in {\mathbb{N}^n}} \beta_{i,\js} x^{\js})=1-\mathcal{R}_I(x)\ ,
\]
and we call the polynomial $\mathcal{R}_I(x)$  the {\em reliability polynomial} of its corresponding system.  
%If the edges operate independently with the same probability$p$, then the reliability of the system is  $\mathcal{R}_I(p)$ in $p$. 
The evaluation of the  reliability polynomial  %$\mathcal{R}_I(p_e)$  
in $p_e$'s  gives us the probability that the system works and
% (in ST setting that there is a path connecting $s$ and $t$). 
its evaluation %of  the reliability polynomial 
in $1-p_e$'s gives us the probability that the system fails. 
%Moreover the Alexander inversion formula  in terms of probability implies that 
% .In ST setting, that means there is no path connecting $s$ and $t$. Moreover, 
%\[\mathcal{R}_{\P_{s,t}}(p_e)=1-\mathcal{R}_{\C_{s,t}}(1-p_e)\ .\]

%\medskip

%\vspace{1mm}
One natural way to describe a resolution of an ideal is through the construction of a polyhedral complex whose faces
% (vertices, edges, and higher dimensional cells) 
are labeled by monomials in such a way that the chain complex determining its cellular homology realizes a graded free resolution of the ideal.   The study of cellular resolutions was initiated by Bayer-Sturmfels in \cite{BayerSturmfels}.

\begin{Theorem}\label{intro:thm1}
The reliability polynomial of the ST system $(G,s,t)$ can be read from a subcompelx  of its associated graphic hyperplane arrangement. 
There is a bijection between the set of faces of this complex and the set of partial acyclic orientations of $(G,s,t)$.
% with a unique source $s$ and a target $t$. 
\end{Theorem}

\begin{Example}\label{exam:double}{\rm
Consider the {\em double bridge} network \cite[Exam 6.2.1]{Dohmen}  in Figure \ref{fig3:graph}(a). 
The set of cuts between $s$ and $t$ is
$\{1258, 24568,2346,123, 1478, 678, 3567,13457\}$. %We call the cuts $p_1,\ldots,p_8$ (keeping the same order). 
Then %the cut ideal of $G$ is
\vspace{-2mm}
\[
\C_{s,t}=\langle x_1x_2x_3, x_2x_3x_4x_6, x_6x_7x_8, x_1x_2x_5x_8, x_1x_4x_7x_8, x_3x_5x_6x_7,
    x_2x_4x_5x_6x_8, x_1x_3x_4x_5x_7 \rangle.
\]
The Betti table of $\C_{s,t}$ is depicted in Figure~\ref{fig3:graph}(b) where its $(i,j)$-entry is simply $\beta_{i,i+j}$. 
%\[
%0\rightarrow  \Rb^4 \rightarrow \Rb^{14}  \rightarrow \Rb^{17}   \rightarrow \Rb^8 \rightarrow \Rb \rightarrow \Rb/\C_{s,t}  \rightarrow 0\ ,
%\]
The inclusion-exclusion expression of the reliability formula contains $2^8-1=255$ terms (with only $43$  non-cancelling terms). 
By Theorem~\ref{thm:DG} these non-cancelling terms are corresponding to $\sum \beta_i(\C_{s,t})=43$ partial acyclic orientations of $G$. %The last Betti number counts the acyclic orientations of $G$ depicted in Figure~\ref{fig3:graph}(c).
The polyhedral complex $\mathcal{D}_G^{s,t}$  in Figure~\ref{fig:poly} supports the minimal free resolution of $\C_{s,t}$.
%, where $\beta_i$ counts the number of $i$-dimensional faces of the complex. 
Its four  facets  are corresponding to the acyclic orientations of $G$ (depicted in Figure~\ref{fig3:graph}(c)).}
Thus for $p_e=p$ we have
\vspace{-1mm}
\[\mathcal{R}_{\C_{s,t}}(p)=(2p^3+4p^4+2p^5)-(4p^5+13p^6)+14p^7-4p^8.\]
\end{Example}
\vspace{.96mm}

\begin{figure}[h]

\begin{center}
\begin{tikzpicture} [scale = .08, very thick = 10mm]

  \node (n4) at (0,-3)  [Cgray] {};
  \node (n1) at (0,15) [Cgray] {};
  \node (n2) at (-11,6)  [Cred] {};

\node (n10) at (-15,6)  [Cwhite] {$s$};
  \node (n11) at (13,6)  [Cwhite] {$t$};

  \node (n3) at (11,6)  [Cblue] {};
\node (n5) at (0,6)  [Cgray] {};
  \foreach \from/\to in {n4/n2,n1/n3}
    \draw[] (\from) -- (\to);
\foreach \from/\to in {n2/n1,n4/n3,n5/n2,n5/n3,n5/n1,n5/n4}
    \draw[] (\from) -- (\to);

 \node (m4) at (-5,8)  [Cwhite] {$2$};  \node (m5) at (5,8)  [Cwhite] {$7$};
   \node (m1) at (-5.5,13) [Cwhite] {$1$};
 \node (m1) at (1,10) [Cwhite] {$4$}; \node (m1) at (1,1) [Cwhite] {$5$};
  \node (m1) at (6,13) [Cwhite] {$6$};
 \node (m1) at (6,0) [Cwhite] {$8$}; 
\node (m1) at (-6,0) [Cwhite] {$3$};
\foreach \from/\to in {n2/n1,n2/n5,n2/n4, n4/n3,n4/n5,n5/n1,n5/n3,n1/n3}
\draw[black][] (\from) -- (\to);

%%%%%%%%%%%%

  \node (n4) at (94,7)  [Cgray] {};
  \node (n1) at (94,25) [Cgray] {};
  \node (n2) at (83,16)  [Cred] {};

  \node (n3) at (105,16)  [Cblue] {};
\node (n5) at (94,16)  [Cgray] {};
  \foreach \from/\to in {n4/n2,n1/n3}
   \foreach \from/\to in {n2/n5,n4/n3,n2/n1,n2/n4,n1/n3,n5/n3}
    \draw[->] (\from) -- (\to);
\foreach \from/\to in {n1/n5,n4/n5}
    \draw[blue][->] (\from) -- (\to);

%%%%%%%%%%%%%%%%

  \node (n4) at (94,-15)  [Cgray] {};
  \node (n1) at (94,3) [Cgray] {};
  \node (n2) at (83,-6)  [Cred] {};

  \node (n3) at (105,-6)  [Cblue] {};
\node (n5) at (94,-6)  [Cgray] {};
  \foreach \from/\to in {n2/n5,n4/n3,n2/n1,n2/n4,n1/n3,n5/n3}
    \draw[->] (\from) -- (\to);
\foreach \from/\to in {n5/n1,n5/n4}
    \draw[blue][->] (\from) -- (\to);

%%%%%%%%%%%%%%

%%Two copies on the right

%%%%%%%%%%%%%%

  \node (n4) at (124,7)  [Cgray] {};
  \node (n1) at (124,25) [Cgray] {};
  \node (n2) at (113,16)  [Cred] {};

  \node (n3) at (135,16)  [Cblue] {};
\node (n5) at (124,16)  [Cgray] {};
  \foreach \from/\to in {n2/n5,n4/n3,n2/n1,n2/n4,n1/n3,n5/n3}
    \draw[->] (\from) -- (\to);
\foreach \from/\to in {n1/n5,n5/n4}
    \draw[blue][->] (\from) -- (\to);

%%%%%%%%%%%%%%%%

  \node (n4) at (124,-15)  [Cgray] {};
  \node (n1) at (124,3) [Cgray] {};
  \node (n2) at (113,-6)  [Cred] {};

  \node (n3) at (135,-6)  [Cblue] {};
\node (n5) at (124,-6)  [Cgray] {};
  \foreach \from/\to in {n2/n5,n4/n3,n2/n1,n2/n4,n1/n3,n5/n3}
    \draw[->] (\from) -- (\to);
\foreach \from/\to in {n5/n1,n4/n5}
    \draw[blue][->] (\from) -- (\to);

%%%%%%%%%%%%%%
\node[draw,align=left] at (47,7) {$\begin{matrix}
      &0&1&2&3&4\\ \text{total:}&1&8&17&14&4\\\text{0:}&1&$-$&$-$&$-$&$-$\\\text{1:}&$-$&
       $-$&$-$&$-$&$-$\\\text{2:}&$-$&2&$-$&$-$&$-$\\\text{3:}&$-$&4&4&$-$&$-$\\\text{4:}&$-$&2&13&14&4\\\end{matrix}
$};

\node (n20) at (4,-19.7)  [Cwhite] {(a)};
\node (n20) at (48,-19.7)  [Cwhite] {(b)};
\node (n20) at (110,-19.7)  [Cwhite] {(c)};
\end{tikzpicture}
\caption{(a) The double bridge network \ (b) The Betti table of $\C_{s,t}$  \ \ \ \quad (c) The acyclic orientations corersponding to the last Betti number of $\C_{s,t}$}
\label{fig3:graph}
\end{center}

\end{figure}
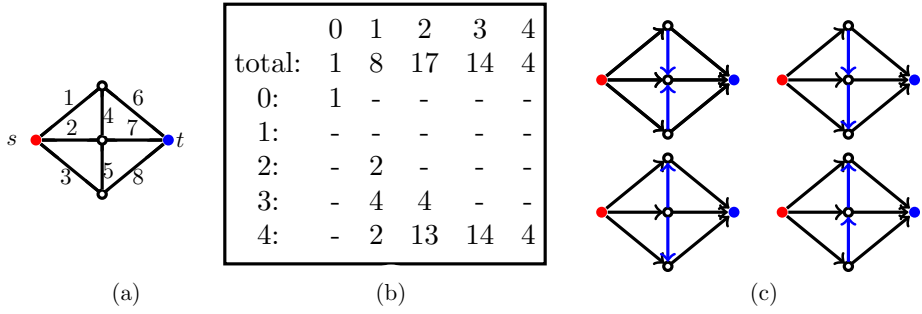

\vspace {-.7cm}

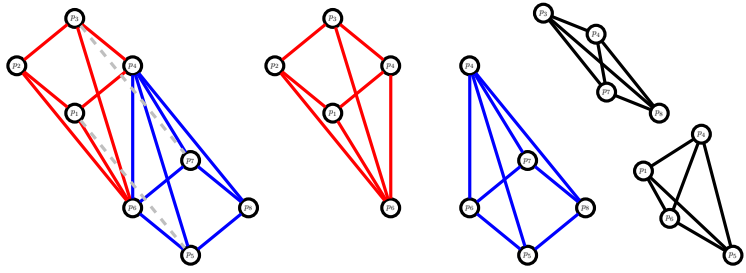
\begin{figure}[h]

\begin{center}

\begin{tikzpicture} [scale = .07, very thick = 10mm]

  \node (n1) at (25,-3)  [Cgray] {$p_1$};
  \node (n3) at (25,15) [Cgray] {$p_3$};
  \node (n2) at (14,6)  [Cgray] {$p_2$};
  \node (n4) at (36,6)  [Cgray] {$p_4$};
%\draw[black][] (\from) -- (\to);

  \node (m1) at (47,-30)  [Cgray] {$p_5$};
  \node (m3) at (47,-12) [Cgray] {$p_7$};
  \node (m2) at (36,-21)  [Cgray] {$p_6$};
  \node (m4) at (58,-21)  [Cgray] {$p_8$};

 \foreach \from/\to in {n2/n1,n2/n3,n3/n4, n4/n1, n1/m2,n2/m2,n3/m2}%,n4/n5,n5/n1,n5/n3,n1/n3}
 \draw[red][] (\from) -- (\to);

 \foreach \from/\to in {n4/m1,n4/m3,n4/m4,m2/m1,m2/m3,m3/m4, m4/m1}%,n4/n5,n5/n1,n5/n3,n1/n3}
 \draw[blue][] (\from) -- (\to);

\foreach \from/\to in {n4/m2}
 \draw[blue][] (\from) -- (\to);

\foreach \from/\to in {n1/m1}
\draw[lightgray][dashed] (\from) -- (\to);

\foreach \from/\to in {n3/m3}
\draw[lightgray][dashed]  (\from) -- (\to);

%\path[fill,color=yellow]
%(34,-3) -- (34,15) -- (23,6); 
%\path[fill,color=yellow]
%(34,-3) -- (34,15) -- (45,6);
 %\path[fill,color=darkgray]
%(-14,-2) -- (-10,-2)--(-10,1)--(-11.6,2.3);

%%%%%%%%%%%%%

  \node (n1) at (74,-3)  [Cgray] {$p_1$};
  \node (n3) at (74,15) [Cgray] {$p_3$};
  \node (n2) at (63,6)  [Cgray] {$p_2$};
  \node (n4) at (85,6)  [Cgray] {$p_4$};
  \node (n5) at (85,-21)  [Cgray] {$p_6$};
%\draw[black][] (\from) -- (\to);

  \node (m4) at (100,6)  [Cgray] {$p_4$};
  \node (m1) at (111,-30)  [Cgray] {$p_5$};
  \node (m3) at (111,-12) [Cgray] {$p_7$};
  \node (m2) at (100,-21)  [Cgray] {$p_6$};
  \node (m5) at (122,-21)  [Cgray] {$p_8$};

 \foreach \from/\to in {n2/n1,n2/n3,n3/n4, n4/n1, n1/n5,n2/n5,n3/n5, n4/n5}%,n4/n5,n5/n1,n5/n3,n1/n3}
 \draw[red][] (\from) -- (\to);

 \foreach \from/\to in {m5/m1,m5/m3,m5/m4,m2/m1,m2/m3,m3/m4, m4/m1,m2/m4}%,n4/n5,n5/n1,n5/n3,n1/n3}
 \draw[blue][] (\from) -- (\to);

%%%%%%%%%%%%%%%%%%%%%traiangle

  \node (m3) at (126,1)  [Cgray] {$p_7$};
  \node (n3) at (124,12) [Cgray] {$p_4$};
  \node (n4) at (114,16)  [Cgray] {$p_3$};
  \node (m4) at (136,-3)  [Cgray] {$p_8$};

\foreach \from/\to in {n3/n4,m3/m4,n3/m3,n3/m4,n4/m3,n4/m4}
\draw[black][]  (\from) -- (\to);

%%%%%%%%%%%%%%%%%%%%%traiangle

  \node (m3) at (138,-23)  [Cgray] {$p_6$};
  \node (n3) at (144,-7) [Cgray] {$p_4$};
  \node (n4) at (133,-14)  [Cgray] {$p_1$};
  \node (m4) at (150,-30)  [Cgray] {$p_5$};

\foreach \from/\to in {n3/n4,m3/m4,n3/m3,n3/m4,n4/m3,n4/m4}
\draw[black][]  (\from) -- (\to);

%%%%%%%%%%%%%%%%%%%%%%%

\end{tikzpicture}
\caption{The polyhedral cell complex $\mathcal{D}_G^{s,t}$ resolving the MFR %minimal %free resolution
 of $\C_{s,t}$}
\label{fig:poly}
\end{center}

\end{figure}

\vspace{3mm}

\subsection{Source-to-multiple-terminal (SMT) system}\label{subsec:SMT}
Another well-known example in the theory of system reliability is the SMT system. 
The set-up is similar to ST configuration. We fix a pointed graph $(G,s)$ with the edge set $E(G)$, and the oriented edge set $\EE(G)$. We study the probability that there exists at least one (oriented) path from $s$ to every other
vertex of $G$.  We let %$k$ be a field, 
$R=k[\xb]$ be the polynomial ring in the variables $\{x_e: e \in E(G)\}$ and 
$S=k[\yb]$ be the polynomial ring in the variables $\{y_e: e \in \EE(G)\}$.
The ideal corresponding to the SMT system is the spanning tree ideal of $G$.
For each spanning tree $T$ of $G$, let $\Oc_T$ denote the orientation of $T$ with a unique source at $s$ (i.e. the orientation obtained by orienting all paths away from $s$), see Figure~\ref{fig:spanning trees1}. 
Any spanning tree $T$ of $G$ gives rise to two monomials
$\yb^T=\prod_{e\in \OO(T)} y_e$ and $\xb^T = \prod_{e \in E(T)}{x_{e}}$.
We define the {\it oriented spanning tree  ideal} $\Tf_G^s$ and the {\it spanning tree ideal} 
$\Tf_G$ as
\[
\Tf_G^s=\{\yb^T :\ T\ {\rm is\ a\ spanning\ tree\ of\ } G\rangle%\subset S
\ \ {\rm and}\
\
\Tf_G=\langle \xb^T:\ T\ {\rm is\ a\ spanning\ tree\ of\ } G\rangle.% \subset R\ .
\] 
The reliability polynomial of the SMT system for a (directed) graph is defined as  
\vspace{-2mm}
\[
\mathcal{R}_{\Tf_G^q}(x)=-\sum_{i=1}^d (-1)^{i+1}(\sum_{\js \in {\mathbb{N}^n}} \beta_{i,\js} x^{\js})\ .
%\ {\rm and}\ 
%\mathcal{R}_{\Tf_G}(x)=-\sum_{i=1}^d (-1)^{i+1}(\sum_{\js \in {\mathbb{N}^n}} \beta_{i,\js} x^{\js})
\]
%is a squarefree monomial ideal generated by all monomials corresponding to spanning trees of $G$.
If $G$ is an undirected graph and $p_e=p$ for all edges, then the reliability formula of the SMT system can be expressed in terms of the Tutte polynomial \footnote{ The Tutte polynomial is  defined by $T(x,y)=\sum_{A\subseteq E}(x-1)^{k(A)-k(E(G))}(y-1)^{k(A)+|A|-|V(G)|}$, where $k(A)$ is the number of connected components of the subgraph on $A$.} $T(x,y)$ of $G$  by
\vspace{-2mm}
\[
\mathcal{R}_{\Tf_G}(p) = (1-p)^{|E(G)|-|V(G)|+1} p^{|V(G)|-1}T(1,\frac{1}{1-p})\ .
\]
%Note that  \eqref{eq:tutte} gives more clear insight into the structure of the system and the edges may work with different probabilities. 
\vspace{-6mm}
\subsubsection{Divisors on graphs}\label{subsec:div}
Let $(G,s)$ be a finite pointed graph, 
and $\Div(G)$ be the free abelian group generated by $V(G)$. An element %$D=\sum_{v\in V(G)} D(v)$ 
of $\Div(G)$ is a formal sum of vertices with integer coefficients and is called a {\em divisor} on $G$. 
We denote by $\M(G)$ the group of integer-valued functions on $V(G)$. The {\em Laplacian operator} $\Delta : \M(G) \to \Div(G)$ is defined by
\vspace{-2mm}
\[\Delta(f) = \sum_{v \in V(G)} \sum_{\{v,w\} \in E(G)} (f(v) - f(w)) (v) .\]
The group of {\em principal divisors} is defined as the image of the Laplacian operator and is denoted by $\Prin(G)$. There is an equivalence relation on the set of divisors. Two divisors $D_1$ and $D_2$ are called {\em linearly equivalent} if:
\begin{equation}\label{eq:equiv2}
 D_1\sim D_2 \Leftrightarrow {\rm \ there\ exists\ a \ function\ } f \ {\rm with\ } D_2=D_1-\Delta(f)
\end{equation}
The set of equivalence classes forms a finitely generated abelian group which is called the {\em Picard group} of $G$. If $G$ is connected, then the finite (torsion) part of the Picard group has cardinality equal to the number of spanning trees of $G$. We recommend the recent survey article \cite{Levine1} for a short overview of the subject.

A divisor $D$ is called $s$-reduced
if % Inclusion-exclusion-bonferroni 
$(i)$ $D(v)\geq 0$ for all $v \in V(G)\backslash\{s\}$ 
$(ii)$ for every non-empty subset $A\subseteq V (G)\backslash \{s\}$, there exists $v\in A$ such that $D(v)<|\EE(A,\{v\})|$, where $\EE(A,\{v\})$ contains the edges directed from a vertex in $A$  toward $v$.
%\end{itemize}
%Note that  for undirected graphs we have that $|\EE(A,\{v\})|=|E(A,\{v\})|$. %for $A\subset V(G)$ and $v\in V(G)\backslash\{v\}$, the number %of edges between $v$ and $A$ is denoted by $\outdeg_A(v)$. 

The $s$-reduced divisors play an important role in divisor theory, because each divisor has a unique equivalent divisor among $s$-reduced divisors (see, e.g., \cite{Dhar90,FarbodMatt12}). %CoriLeBorgne03,  CoriRossinSalvy02, BN1, }).  

\begin{Theorem}\label{intro:thm2}
There is a bijection between non-cancelling terms in the reliability polynomial of the SMT system  $(G,s)$, the set of multigraded Betti numbers of its spanning tree ideal,
the set of oriented $k$-spanning trees, and  $s$-reduced divisors of degree $k-1$ with $D(s)=-1$.
%, and the minimal generating set of the $k$-syzygy module of $\Tf_G$. %set of $k$ compute the $q$-reduced divisors of $G$. 
%Furthermore, the non-cancelling terms  %of $\mathcal{R}_{\Tf_G^s}(x)$ 
%can be read from MFR of the toppling ideal $I_G$ and its initial ideal $\MM_G^s$.
\end{Theorem}

%%%%%%%%%%%%%%%%%%%%%%%%%%%%%%%%%%%%%%%%%%%%%%%%%%%%%%%%%%%%%%%%%%%%%%%%%%%%%%%%%%%%%%%%%%%%%%%%%%%%%%%%%%%%%%%

\section{Partial orientations and  divisors on graphs}\label{sec:reduced}
%The goal of 
In this section we provide a short overview of the state-of-the-art on acyclic orientations.  
\smallskip
From now on we fix a graph $G$ and we let $n=|V(G)|$. 
Let $\EE(G)$ denote the set of oriented edges of $G$; for each edge in $E(G)$ there are two edges $e$ and $\bar{e}$ in $\EE(G)$. So we have $|\EE(G)|=2m$. 
An element $e$ of $\EE(G)$ is called an {\em oriented} edge, and $\bar{e}$ is called the {\em inverse} edge. 
We have a map 
sending an oriented edge $e$ to its head (or its terminal vertex) $e_{+}$ and its tail (or its initial vertex) $e_{-}$. Note that $\bar{e}_{+}=e_{-}$ and $\bar{e}_{-}=e_{+}$. 
An {\em orientation} of $G$ is a choice of subset $\Oc \subset \EE(G)$ such that $\EE(G)$ is the disjoint union of $\Oc$ and $\bar{\Oc}=\{\bar{e}: \ e \in \Oc \}$. An orientation is called {\em acyclic} if it contains no directed cycle.
A {\em partial orientation} of $G$ is a choice of subset $\P \subset \EE(G)$ that is strictly contained in an orientation $\Oc$ of $G$. 
A partial orientation is called {\em acyclic} if the induced orientation on the graph obtained by contracting all its components is acyclic.

\medskip

Let $\Oc$ be an orientation of $G$. A vertex $s$ is called a {\em source} for $\Oc$ if $s=e_-$ for every $e \in \Oc$ which is incident to $s$.
Let $\P$ be a partial orientation of $G$, and let $H$ be the associated connected component containing the vertex $s$. 
Then $H$ is called a {\em source} for $\P$, if $H$ corresponds to a source in the graph obtained by contracting all components of $\P$.
We define $D_{\P}$ to be the divisor associated to $\P$ with $D_{\P}(v)=\indeg_{\P}(v)-1$, where $\indeg_{\P}(v)$ denotes the number of oriented edges directed {\em to} $v$ in $\P$.

\vspace{1mm}

Given disjoint nonempty subsets $A,B$ of $V(G)$ we define 
\[
\EE(A,B) = \{e \in \EE(G): e_+ \in A ,\ e_- \in B\}.
\]
To the ordered pair $(A,B)$ we  assign the effective divisor 
$D(A,B)=\sum_{v\in A}|\EE(\{v\},B)|(v).$

\medskip

For a nonempty subset $A$ of $V(G)$, $\EE(A,A^c)$ is called a {\em cut} of $G$, and any proper subset $\C$ of $\EE(A,A^c)$ is called a {\em partial cut} of $G$.
A cut $\C=\EE(A,A^c)$ is called {\em connected} if $G[A]$ and $G[A^c]$ are connected.
For each cut $\C=\EE(A,A^c)$ of $G$, we define its {\em inverse} as $\bar{\C}=\{\bar{e}:\ e\in \C\}$. The result of  applying a cut-inverse 
operation on a 
partial orientation $\P$ and the cut $\EE(A,A^c)$, is the partial orientation $\P'$ which only inverses the edges of $\EE(A,A^c)$, but preserves the other edges of $\P$ unchanged. 
Thus the only difference between $\P$ and $\P'$ is on the edges of $\EE(A,A^c)$. In other words, 
$\P'={\EE(A^c,A)}\cup(\P\backslash \EE(A,A^c))$.
Similarly one can define the inverse of a cycle in a partial orientation $\P$. The operation inverting a cut in $\P$ is called a {\em cut reversal}, and the operation inverting a cycle is called a {\em cycle reversal} (see \cite{Gioan} for more details). 
\begin{Definition}\label{def:equiv}
Two partial orientations $\P$ and $\P'$ are called 
{\em equivalent in cut-cycle reversal system}  if there exists a sequence $\P=\P_1,\P_2,\ldots,\P_k=\P'$ of orientations such that for each $i$, $\P_{i+1}$ is obtained from $\P_i$ by inverting a cut or a cycle in
$\P_i$ (see Figure~\ref{fig:cut}).
\end{Definition}
 
\vspace{-3mm}

\begin{figure}[h!]  \begin{center}

\begin{tikzpicture}  [scale = .13, very thick = 15mm]

    %%%%%%%%%%%%%%%%%%%%%%%%%%%%%%
    
  \node (n1) at (5,11) [Cgray] {};
  \node (n2) at (1,6)  [Cgray] {};
  \node (n3) at (9,6)  [Cgray] {};
  \node (n4) at (3,1)  [Cgray] {};
  \node (n5) at (7,1)  [Cgray] {};
  \foreach \from/\to in {n1/n3,n2/n4}
    \draw[blue][->] (\from) -- (\to);
\foreach \from/\to in {n3/n4,n4/n5,n5/n3}
    \draw[->] (\from) -- (\to);
\foreach \from/\to in {n1/n2}
    \draw[] (\from) -- (\to);

    %%%%%%%%%%%%%%%%%%%%%%%
 
  \node (n1) at (20,11) [Cgray] {};
  \node (n2) at (16,6)  [Cgray] {};
  \node (n3) at (24,6)  [Cgray] {};
  \node (n4) at (18,1)  [Cgray] {};
  \node (n5) at (22,1)  [Cgray] {};
    \foreach \from/\to in {n3/n4,n4/n5,n5/n3}
    \draw[->] (\from) -- (\to);
\foreach \from/\to in {n1/n2}
    \draw[] (\from) -- (\to);
\foreach \from/\to in {n1/n3,n2/n4}
    \draw[red][<-] (\from) -- (\to);

    %%%%%%%%%%%%%%%%%%%%%%%
 
  \node (n1) at (35,11) [Cgray] {};
  \node (n2) at (31,6)  [Cgray] {};
  \node (n3) at (39,6)  [Cgray] {};
  \node (n4) at (33,1)  [Cgray] {};
  \node (n5) at (37,1)  [Cgray] {};
  \foreach \from/\to in {n3/n4,n4/n5,n5/n3}
    \draw[blue][->] (\from) -- (\to);
\foreach \from/\to in {n1/n2}
    \draw[] (\from) -- (\to);
\foreach \from/\to in {n1/n3,n2/n4}
    \draw[black][<-] (\from) -- (\to);

 %%%%%%%%%%%%%%%%%%%%%%%
 
  \node (n1) at (50,11) [Cgray] {};
  \node (n2) at (46,6)  [Cgray] {};
  \node (n3) at (54,6)  [Cgray] {};
  \node (n4) at (48,1)  [Cgray] {};
  \node (n5) at (52,1)  [Cgray] {};
   \foreach \from/\to in {n3/n4,n4/n5,n5/n3}
    \draw[red][<-] (\from) -- (\to);
\foreach \from/\to in {n1/n2}
    \draw[] (\from) -- (\to);
\foreach \from/\to in {n1/n3,n2/n4}
    \draw[black][<-] (\from) -- (\to);

\end{tikzpicture} \end{center}

\caption{$\P_1,\P_2,\P_2,\P_3$:\ where $\P_2$ is obtained by a cut-inverse from $\P_1$ and $\P_3$ by a cycle-inverse from $\P_2$.}
\label{fig:cut}
\end{figure}
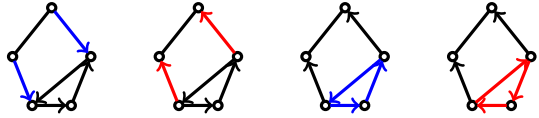
\vspace{-2mm}
%%%%%%%%%%%%%%%%%%%%%%%%%%%%%%%%%%%%%

The following  special  partial orientations of $G$  arise naturally in our setting.

%%%%%%%%%%%%%%%%%%%%%%%%%%%%%%%%%%%%%%%%%%%%%%%%%%%%%%%%%%%%%%%%%%%%%%%%%%%%%%%%%%%%%%%%%%%%%%%%%%%%%%%%%%%%%%%

\begin{Remark}\label{rem:un/indirect}
It is easy to find two partial orientations having identical associated divisors. For example, let $e\in \P$ and $e'\not\in{\P}$ with $e_+=e'_+$. Then for $\P'=\{e'\}\cup\P\backslash\{e\}$ we have that $D_{\P}=D_{\P'}$. Moreover, 
$\P'\cap \EE(A^c,A)\subsetneq \P\cap \EE(A^c,A)$ for any $A$, with $e'_-, e'_+\in A^c$ and $e_-\in A$. We write $\P\sim_1 \P'$ if there is a sequence of moves, taking $\P$ to $\P'$ by exchanging pair of edges in each step, as explained.
So we slightly modify Definition~\ref{def:equiv} as follows:
\begin{eqnarray}\label{eq:equiv1}
\P\sim \P' \Leftrightarrow{\rm\ there\ exists\ a\ sequence\ }\P=\P_1,\P_2,\ldots,\P_k=\P',{\rm\ where}
\end{eqnarray}
for each $i$, $\P_{i+1}$ is a partial orientation such that $\P_{i+1}\sim_1 \P_i$, or it is obtained from $\P_i$ by inverting a cut, or a cycle.%, or substituting a directed edge $e$ with $e'$, where $e'\not\in\P_i$ and $e_+=e'_+$. %as Remark~\ref{rem:un/indirect}.
\end{Remark}
\vspace{-2mm}

\begin{Example}
Let $G$ be the following graph on the vertices $v_1,v_2,\ldots,v_5$. We fix $A=\{v_4\}$. We start from $\P$, and in each step, we substitute the red edge $e$ with the blue edge directed to $e_+$, to obtain a new orientation.
Then, as we see, their associated divisors coincide, and $|\P_1\cap\EE(A^c,A)|=2$, $|\P_2\cap\EE(A^c,A)|=1$ and $|\P_3\cap\EE(A^c,A)|=0$.

\vspace{-3mm}

\begin{figure}[h!]  \begin{center}

\begin{tikzpicture}  [scale = .13, very thick = 15mm]

 \node(p1) at (3.5, 11.5) [C0] {$v_1$};
    \node(p2) at (-0.5, 6.5) [C0] {$v_2$};
        \node(p3) at (10.5, 6.5) [C0] {$v_3$};
    \node(p4) at (1.5, 0.5) [C0] {${v_4}$};
    \node(p5) at (8.5, 0.5) [C0] {$v_5$};

    %%%%%%%%%%%%%%%%%%%%%%%%%%%%%%
    
  \node (n1) at (5,11) [Cgray] {};
  \node (n2) at (1,6)  [Cgray] {};
  \node (n3) at (9,6)  [Cgray] {};
  \node (n4) at (3,1)  [Cgray] {};
  \node (n5) at (7,1)  [Cgray] {};
  \foreach \from/\to in {n5/n4,n4/n3,n4/n2,n5/n3}
    \draw[->] (\from) -- (\to);
\foreach \from/\to in {n1/n2,n1/n3}
    \draw[] (\from) -- (\to);

 %\node(p1) at (3.5, 11.5) [C0] {$v_1$};
  %  \node(p2) at (-0.5, 6.5) [C0] {$v_2$};
   %     \node(p3) at (10.5, 6.5) [C0] {$v_3$};
    %\node(p4) at (1.5, 0.5) [C0] {${v_4}$};
    %\node(p5) at (8.5, 0.5) [C0] {$v_5$};

    %%%%%%%%%%%%%%%%%%%%%%%
 
  \node (n1) at (20,11) [Cgray] {};
  \node (n2) at (16,6)  [Cgray] {};
  \node (n3) at (24,6)  [Cgray] {};
  \node (n4) at (18,1)  [Cgray] {};
  \node (n5) at (22,1)  [Cgray] {};
    \foreach \from/\to in {n5/n4,n4/n3,n5/n3}
    \draw[->] (\from) -- (\to);
\foreach \from/\to in {n1/n3}
    \draw[] (\from) -- (\to);

 \foreach \from/\to in {n1/n2}
    \draw[blue] (\from) -- (\to);
\foreach \from/\to in {n4/n2}
    \draw[red][->] (\from) -- (\to);

    %%%%%%%%%%%%%%%%%%%%%%%
 
  \node (n1) at (35,11) [Cgray] {};
  \node (n2) at (31,6)  [Cgray] {};
  \node (n3) at (39,6)  [Cgray] {};
  \node (n4) at (33,1)  [Cgray] {};
  \node (n5) at (37,1)  [Cgray] {};
  \foreach \from/\to in {n5/n4,n1/n2,n5/n3}
    \draw[->] (\from) -- (\to);
\foreach \from/\to in {n4/n2}
    \draw[] (\from) -- (\to);
 
 \foreach \from/\to in {n1/n3}
    \draw[blue][-] (\from) -- (\to);
\foreach \from/\to in {n4/n3}
    \draw[red][->] (\from) -- (\to);

 %%%%%%%%%%%%%%%%%%%%%%%
 
  \node (n1) at (50,11) [Cgray] {};
  \node (n2) at (46,6)  [Cgray] {};
  \node (n3) at (54,6)  [Cgray] {};
  \node (n4) at (48,1)  [Cgray] {};
  \node (n5) at (52,1)  [Cgray] {};
  \foreach \from/\to in {n1/n3,n5/n4,n1/n2,n5/n3}
    \draw[->] (\from) -- (\to);
\foreach \from/\to in {n4/n2,n3/n4}
    \draw[] (\from) -- (\to);

\end{tikzpicture} \end{center}

\caption{$\P, \P_1,\P_2,\ {\rm and\ }\P_3$}
\label{fig:cut2}
\end{figure}
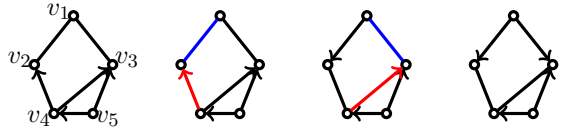
\end{Example}

\vspace{-4mm}
This example motivates the following definition.

\begin{Definition}\label{def:prec1}
Fix a subset $A\subset V(G)$ and the orientation $\P$ of $G$. The set of all partial orientations $\P'$ of $G$ with $D_{\P'}=D_{\P}$, will be denoted by $S(\P)$.
Let $\preceq_{(A,\P)}$ denote the ordering on the elements of $S(\P)$, %with the associated divisor $D_{\Oc}$,
given by reverse inclusion:
\[
\P'' \preceq \P' \iff \P'\cap\EE(A^c,A)\subsetneq \P''\cap\EE(A^c,A) \ .
\]
We would like to find such partial orientations $\P'$, when $\P'\cap\EE(A^c,A)$ has the smallest possible size, and among them, we consider those with the maximum number of oriented edges from $A^c$ to $A$. 
We fix, once and for all, a {\em total ordering} extending $\preceq_{(A,\P)}$. 
By a slight abuse of notation, $\preceq_A$ will be used to denote this total ordering extension. In particular $\prec_A$ will denote the associated strict total order. 
We denote the {\em maximal element} of $S(\P)$ (with respect to $\prec_A$) with $\P_A$.
\end{Definition}

%%%%%%%%%%%%%%%%%%%%%%%%%%%%%%%%%%%%%%%%%%%%%%%%%%%%%%%%%%%%%%%%%%%%%%%%%%%%%%%%%%%%%%%%%%%%%%%%%%%%%%%%%%%%%%%

%\subsection{Main properties of $[\P]$ and $\P_A$} \label{sec:technical}

\begin{Lemma}\label{lem:cut}
Fix a subset $A\subset V(G)$ and a partial orientation $\P$ of $G$. %Then the following hold: % and $\P_A$ be the partial orientation of $G$ (given by Definition~\ref{def:prec1}).
Then there exist a cut $\C$ and a partial orientation $\P'$ such that $\C\subset \P'$ and $\P'\sim_1 \P_A$. If $|\EE(A^c,A)\cap \P_A|>0$, then we can choose $\C$ with  $|\EE(A^c,A)\cap \C|>0$.
\end{Lemma}

\begin{proof} We consider two different cases:

(i) $|\EE(A^c,A)\cap \P_A|>0$: Assume that $e'\in \EE(A^c,A)$. Then consider the set 
 \[
 C=\{{e'}_{-}\}\cup\{e_+:\ e_+\in A^c\ {\rm and\ there\ is\ a \ path}\ e'=e_1,e_2,\ldots,e\ {\rm in\ }\P_A\}\ ,
 \]
and let $\C=\EE(C,C^c)$. It is clear that $|\EE(C^c,C)|=0$, since otherwise the vertex $e_+$ corresponding to $e\in \EE(C^c,C)$ would be in $C$, as well. By contrary assume that $\C\subsetneq \EE(C,C^c)$, say $e\in \EE(C,C^c)\backslash \P_A$. Consider the path $e_1,e_2,\ldots,e_k$ with $(e_k)_+=e_+$. Then 
the partial orientation $\P'=\{e\}\cup(\P_A\backslash\{e_k\})$ belongs to $S(\P_A)$. By continuing the same procedure, we keep moving an unoriented edge closer to $e_1$ so that we can {\em unorient} the edge $e_1$ and add another oriented edge in $G[C^c]$. 
This way the associated 
divisor will not be changed, however for the new orientation $\P'$ we have $\P'\cap \EE(A^c,A)=(\P_A\cap \EE(A^c,A))\backslash \{e_1\}$, a contradiction (to the choice of $\P_A$).

(ii) $|\EE(A^c,A)\cap \P_A|=0$: First, note that each edge between $A$ and $A^c$, either is undirected, or it is directed from $A^c$ to $A$. 
If $\EE(A,A^c)\subset \P_A$, then $\EE(A,A^c)$ is already a cut. Otherwise, there is an unoriented edge between $A$ and $A^c$.  We may use this edge, and apply the same argument as proof of (i)
(by moving this undirected edge inside a set, and adding a directed edge to $\P_A$) to get  a contradiction as in (i).
\end{proof}

Here we show that 
%The following proposition  shows that 
two equivalence classes \eqref{eq:equiv1} and \eqref{eq:equiv2} are intimately related. 
%In fact they are equal, i.e.,  two orientations are equivalent in the (modified) cut-cycle reversal system, %in the sense of \eqref{eq:equiv1},
% if and only if their corresponding divisors are linearly equivalent.

\begin{Proposition}\label{prop:equiv}
$
\P\sim \P' \Leftrightarrow D_\P\sim D_{\P'}.
$
%where the equivalence relations refer respectively to \eqref{eq:equiv1} and \eqref{eq:equiv2}.
%where the first equivalence relation refers to \eqref{eq:equiv1} and the second refers to \eqref{eq:equiv2}.
\end{Proposition}

\begin{proof}
Let $\P$ be a partial orientation. Exchanging a pair of edges as Remark~\ref{rem:un/indirect}, and also inverting a cycle in $\P$, keep the associated divisor unchanged. 
On the other hand, by inverting a cut in $\P$ we obtain an orientation whose associated divisor is linearly equivalent to $D_\P$ as in \eqref{eq:equiv2}. Thus $\P\sim\P'$ implies that $D_{\P}\sim D_{\P'}$. 

Now assume that $D_\P\sim D_{\P'}$. Then by \eqref{eq:equiv2} there exists (an integer valued) function $f$ with $D_{\P'}=D_{\P}-\Delta(f)$. The proof is by induction on $|\supp(f)|$. First assume that $D_{\P'}=D_{\P}$. Then we show that 
one can obtain $\P'$ from $\P$, only by performing cycle-inverse, and exchanging pair of oriented edges as Remark~\ref{rem:un/indirect}.
The proof is by reverse induction on $|\P\cap \P'|$. Assume that $e\in \P\backslash \P'$. Since $\indeg_{\P}({e}_{+})=\indeg_{\P'}({e}_{+})$, there exists an edge $e'\in\P'\backslash\P$ with $e'_+=e_+$. 
Now we consider two cases, and in each case we find a third orientation $\P''$ such that $\P\cap \P'$ is a proper subset of $\P\cap\P''$ and $\P'\cap\P''$. Therefore, the result follows by induction  hypothesis. 
\medskip

Case 1. $\bar{e}\not\in\P'$ or $\bar{e'}\not\in\P$: If $\bar{e}\not\in\P'$, then we let $\P''$ be the orientation obtained from $\P$ by replacing $e$ with $e'$ as Remark~\ref{rem:un/indirect}. 
If  $\bar{e'}\not\in\P$, then we perform the similar operation (replacing $e'$ with $e$) on $\P'$. 

Case 2. $\bar{e}\in\P'$ and $\bar{e'}\in\P$: %note that either $e'$ is unoriented in $\P$ or ${\bar{e'}}\in\P$. 
Since $\indeg_{\P'}(e'_-)=\indeg_{\P}(e'_-)$, there is $e_1\in \P'\backslash \P$ with $({e_1})_{+}=e'_-$ such that either $e_1$ is undirected in $\P$, or its inverse is in $\P$. In the first case, we are in Case(1). 
Otherwise, by continuing the same argument, we keep moving along a path $\bar{e},{e'},e_1,\ldots$ in $\P'$ and along its inverse 
%path $e,\bar{e'}\bar{e_1}\ldots$ 
in $\P$ which will be terminated at some point. So this way, we can create an oriented cycle. Now inverting this cycle in $\P$ we obtain $\P''$ which is either equal to $\P'$, or it has more intersection with both $\P$ and $\P'$, as desired.

Now assume that $|\supp(f)|>0$ and let $A=\supp(f)$. By Definition~\ref{def:prec1} we may assume that $\P=\P_A$. Note that $D_\P=D_{\P_A}$.
If $|\EE(A^c,A)\cap \P|>0$, by Lemma~\ref{lem:cut} there exists a cut $\C=\EE(C,C^c)\subset \P$ with $|\EE(A^c,A)\cap \C|>0$.
This implies that at least $|\EE(A^c,A)\cap \C|$ edges are directed from $A$ to $C$, and for each $e$ in $\EE(A^c,A)\cap \C$, the indegree of $e_+$ in $\P'$ is greater than its indegree in $\P$. Therefore $D_{\P'}(e_+)$ is greater  than $D_{\P}(e_+)$, a contradiction. 
Thus we have 
$|\EE(A^c,A)\cap \P|=0$. If $\EE(A,A^c)\subset \P$, i.e. $\EE(A,A^c)$ is a cut, then we can inverse this cut to obtain $\P''$. By applying the induction assumption on the support of the function taking $D_{\P''}$ to $D_{\P}$, we conclude that $\P''$ is equivalent to $\P$ and to $\P'$, as desired. Otherwise,  applying 
Lemma~\ref{lem:cut} we obtain a cut $\C'\subset \P$, and we process the cut-inverse corresponding to $\C'$ in order to use the induction hypothesis. 
\end{proof}

\begin{Remark}
Note that an acyclic partial orientation can not be equivalent to an orientation without any source. If %Once we know that 
$\Oc$ does not have any source, then  % it follows that 
any vertex belongs to a directed cycle. On the other hand, even if we can find a cut  $\EE(A,A^c)\subseteq \Oc$  to perform a cut-inverse, the vertices of $A$ and $A^c$  still belong to a cycle in the obtained orientation.  
\end{Remark}

\section{Graphic matroid ideals}
Here, we quickly recall some basic notions from %oriented 
matroid theory. Our main goal is to fix our notation. A secondary goal is to keep the paper self-contained. Most of the material here is well-known and we refer to 
\cite{MR712251,novik,FatemehFarbod2} for proofs and more details.  %In this section, we first recall some definitions related to ideals arising in matroid theory, %and also related to $\M_G^s$ from \S\ref{subsec:div}. 
An important feature that we want to emphasize %in this section
is the relation between the ideals associated to an {\em undirected} network, their corresponding {\em oriented} versions, and their {\em Alexander duals}.

\subsection{Oriented matroid ideals}\label{subsec:oriented-matroid}
An {\em oriented hyperplane arrangement} is a real hyperplane arrangement along with a choice of a {\em positive side} for each hyperplane. Equivalently, one may fix a set of linear forms vanishing on hyperplanes to fix the {\em orientation}. 
For each (oriented) hyperplane arrangement $\{\H_1,\ldots, \H_m\}$ with hyperplanes $\H_j = \{ \vb \in \RR^{n-1}: h_j(\vb)=c_j\}$ living in $\RR^{n-1}$, one can consider a central hyperplane arrangement
$\A=\{\H_1,\ldots,\H_m, \H_g\}$ in $\RR^{n}=\RR^{n-1}\times \RR$ such that
\[
\H_j = \{ (\vb,w) \in \RR^{n-1}\times \RR : h_j(\vb)=c_jw\} \ {\rm\ and\ }\ \H_g=\{ (\vb,w) \in \RR^{n-1}\times \RR : w=0\}.
\] 
Now if we restrict ourselves to the {\em positive side} of $\H_g$, more precisely, consider the restriction of $\A$ to the hyperplane 
$\{(\vb,w) \in \RR^{n-1}\times \RR: w=1\}$ we obtain an affine hyperplane arrangement. 
Let $\Rb=k[\xb]$ be the polynomial ring in $m$ variables $\{x_i: \H_i \in \A\}$ and $\Sb=k[\xb,\yb]$ be the polynomial ring in $2m$ variables $\{x_i,y_i: \H_i\in\A\}$.
For any (affine) oriented hyperplane arrangement one can define (see \cite{novik}) the associated {\em oriented matroid ideal}: let $\{h_j\}$ be $m$ nonzero linear forms defining the hyperplane arrangement $\A$ with hyperplanes 
$\H_j = \{ \vb \in \RR^{n-1} : h_j(\vb)=c_j\}$ in $\RR^{n-1}$. The oriented matroid ideal associated to $\A$ is the ideal in $2m$ variables of the form:
\[
\OO = \langle \mb(\vb) : \vb \in \RR^{n-1} \rangle \subset K[\xb,\yb]\ ,
\]
where for each $\vb \in \RR^{n-1}$
\[
\mb(\vb)=\prod_{h_i(\vb)>c_i}x_i  \prod_{h_i(\vb)<c_i}y_i \ , %{\rm\   where\ } x_g=y_g=1 \ {\rm in\ the\ affine\ case}.
\]
the multiplication being over all $i=1,\ldots,m$. Note that any two points in the relative interior of a cell will give rise to the same monomial.
Moreover, the hyperplanes $\H_1,\ldots, \H_m,\H_g$ partition $\mathbb{R}^{n-1}$ into relatively open convex polyhedra
called the cells of the corresponding arrangement. The cells of dimension zero are called {\em vertices}.
A cell is called a {\em bounded cell} if it is bounded as a subset of $\mathbb{R}^{n-1}$, and we denote $\B_{\mathcal{A}}$ for the set consisting of all bounded
cells of $\mathcal{A}$ which is a regular CW-complex (see e.g., \cite{novik}).

 There is a canonical surjective $k$-algebra homomorphism 
$
\phi \colon \Sb \rightarrow \Rb
$
defined by sending $x_i$ and $y_i$ to $x_{i}$ for all $i$. The kernel of this map is precisely the ideal generated by $\{x_1-y_1,\ldots,x_m-y_m\}$, 
  which we denote by $\aa$. 
The induced isomorphism 
$
\bar{\phi} \colon \Sb /\aa \xrightarrow{\sim} \Rb
$
is the {\em algebraic indegree map}, and it relates the ideals $\OO$ to the ideal 
\[
{\bar{\OO}}=\langle \bar{\mb}(\vb)= \prod_{h_i(\vb)\neq c_i}x_i: \vb \in \RR^{n-1} \rangle \subset \Rb\ .
\]

Before stating our next result, we recall that for a squarefree monomial ideal
$I=\langle \xb^{a_1},\ldots, \xb^{a_r} \rangle\subset k[\xb]$  its {\em Alexander dual}
is defined by $I^\vee= {\mm}^{a_1}\cap\cdots\cap  {\mm}^{a_r}$,  where  $\xb^a=\prod_{i\in a} x_i$ and ${\mm}^{a}=\langle z_i:\ i\in a\rangle$
 for each vector $a\in \mathbb{N}^n$.

\begin{Proposition}\label{prop:primary} 
Let $\mathcal{V}$ be the set consisting of the vertices  of the bounded complex $\B_{\mathcal{A}}$, and for $\vb\in \mathcal{V}$, let $\bar{P}_{\vb}=\langle x_i:\ {h_i(\vb)\neq c_i} \rangle$. Then 
the minimal prime decomposition of the Alexander dual of the ideal $\bar{\OO}$ is $\bar{\OO}^\vee=\bigcap_{\vb\in \mathcal{V}} \bar{P}_{\vb}$. 
\end{Proposition}

\begin{proof}
First note that from the Alexander duality definition, the minimal prime decomposition of $\OO^\vee$ is $\OO^\vee=\bigcap_{\vb\in \mathcal{V}} P_{\vb}$, where 
$P_{\vb}=\langle x_i:\ {h_i(\vb)>c_i} \rangle+\langle y_j:\ {h_j(\vb)<c_j} \rangle$.
Also note that for any $i$ and $\OO_i^{\vee}=\OO^\vee\tensor \Sb/(x_1-y_1,\ldots,x_i-y_i)$, the minimal primes of $\OO_i^{\vee}$
are obtained from the minimal primes of $\OO^\vee$ by specializing the variable $y_\ell$ to $x_\ell$ for all $\ell\leq i$. 
Let $\bar{P}_{\vb}$ denote the ideal obtained from $P_\vb$ by identifying the variables $x_\ell$ and $y_\ell$ for each $\ell\leq i$.
 It is easy to see that $\OO_i^{\vee}\subseteq \bigcap \bar{P}_{\vb}$.

Assume that $m$ is a monomial in $\bigcap \bar{P}_{\vb}$. We want to show that $m\in \OO_i^{\vee}$. 
Let $x_\ell\in \supp(m)$ and  $m=x_\ell m'$ for some $m'$. Thus $x_\ell y_\ell m'\in P_{\vb}$ , and so $x_\ell m'$ or $y_\ell m'$ belongs to $P_{\vb}$, since each prime ideal contains at most one of the variables $x_\ell$ or $y_\ell$. 
However for each $\vb$ the images of the both monomials $x_\ell m'$ and $y_\ell m'$ in $\bar{P}_\vb$ are equal to $m$. Therefore $m\in \OO_i^{\vee}$. 
If $x_\ell\not\in \supp(m)$, then $m\in \bar{P}_{\vb}$ shows that $m\in P_\vb$. This implies that $m$ belongs to $\bigcap P_{\vb}$ and so to $\OO^{\vee}$. 
Thus after identifying  $x_\ell$ and $y_\ell$ we obtain again $m$ (since $x_\ell,y_\ell\not\in\supp(m)$) which belongs to $\OO_i^{\vee}$.
\end{proof}

%%%%%%%%%%STANLEY REISNER

\begin{Remark}\label{rem:CM}
% The simplicial complexes $\Delta_{\OO}$ and $\Delta_{\bar{\OO}}$ are shellable, and so their corresponding Stanley-Reisner 
$(i)$ The rings $\Sb/\OO$ and $\Rb/\bar{\OO}$ are Cohen-Macaulay. % of dimension $2m-{\rm rank}(\mathcal{M}\backslash g)$.
%he simplicial complex $\Delta_{\bar{\OO}}$ is shellable and  $K[\Delta_{\bar{\OO}}]$ is Cohen-Macaulay.
% of dimension $2n-r$, where $r={\rm rank}(\mathcal{M}\backslash g)$.
Moreover, % the $\ZZ$-graded (and
 the multigraded Betti numbers of $\OO$ and ${\bar {\OO}}$ coincide (see \cite[Thm. 10.3]{FatemehFarbod2}).

$(ii)$ By \cite[Thm. 3]{Eagon} the minimal free resolutions of the ideals $\OO^{\vee}$ and ${\bar{\OO}}^{\vee}$ are linear.
\end{Remark}

\begin{Proposition}\label{lem:sameBetti}
$(i)$ The set
$\{x_1-y_1,\ldots, x_m-y_m\}$
forms a regular sequence for $\Sb/\OO^\vee$. 

$(ii)$ The $\ZZ$-graded (and multigraded) Betti numbers of $\OO^\vee$ and ${\bar {\OO}}^\vee$ coincide.
\end{Proposition}
\begin{proof}
(i) By contrary assume that $x_1-y_1$ is a zerodivisor element modulo $\OO^\vee$.  
Then there exists $\vb$ such that $x_1-y_1$, and so $x_1,y_1$ belongs to $P_{\vb}$ which is a contradiction, since at most one of the variables $x_1,y_1$ belongs to $P_\vb$.
Now we show that $x_2-y_2$ is nonzerodivisor modulo $\OO_1^{\vee}=\OO^{\vee}\tensor \Sb/(x_1-y_1)$. Note that by Proposition~\ref{prop:primary} the prime components of $\OO_1^{\vee}$ are obtained by identification $x_1=y_1$ in prime components 
$\OO^{\vee}$. If $x_2-y_2$ is a zerodivisor element modulo $\OO_1^\vee$, then  there exists $\vb$ such that $x_2,y_2$ belong to $\bar{P}_\vb$. Thus $x_2,y_2$ belong to a prime component of $\OO_1^{\vee}$, and so to $P_\vb$.
This contradicts by the fact that each prime component of $\OO^{\vee}$ contains at most one of the variables $x_2$ or $y_2$.
Then continuing the same argument, we show that $x_i-y_i$ is a nonzerodivisor modulo $\OO_{i-1}^\vee=\OO^\vee\tensor \Sb/(x_1-y_1,\ldots,x_{i-1}-y_{i-1})$ for all $i$.

(ii) The result follows by \cite[Lem.~3.15]{EisenbudSyz} (see also \cite[Prop.~1.1.5]{Bruns}).
%, since $\{x_1-y_1,\ldots, x_m-y_m\}$ forms a regular sequence for $\Sb/\OO^\vee$. 
\end{proof}

\subsection{Graphic hyperplane arrangements}\label{sec:graphic}
\label{sec:BG} 
We fix a pointed graph $(G,s)$ on the vertex set $[n]$ with the edge set $E(G)$. 
Following \cite{MR712251}, we define the {\em graphic hyperplane arrangement} as follows. This arrangement lives in the Euclidean space $C^0(G,\RR)$, i.e. the vector space of all real-valued functions on $V(G)$ endowed with the bilinear form 
$
\langle f_1, f_2\rangle = \sum_{v\in V(G)}{f_1(v)f_2(v)}.
$
Let $C^1(G, \RR)$ be the vector space of all real-valued functions on $\EE(G)$, and let $\partial^\ast : C^0(G,\RR) \rightarrow C^1(G,\RR)$ denote the usual coboundary map.
For each edge $e \in \EE(G)$ let $\H_{e} \subset C^0(G,\RR)$ denote the hyperplane
\[
\H_{e} = \{f \in C^0(G,\RR): (\partial^\ast f) (e)=0 \} \ .
\]
Consider the arrangement 
$
\H'_G= \{\H_e : e \in \EE(G)\} 
$
in $C^0(G,\RR)$. Since $G$ is connected, we know $\bigcap_{e \in \EE(G)}{\H_e}$ is the $1$-dimensional space of constant functions on $V(G)$, which is the same as the kernel of $d$. We define the {\em graphic arrangement} corresponding to $G$, denoted by $\H_G$, to be the restriction of $\H'_G$ to the hyperplane 
%\vspace{-1mm}
\begin{equation}\label{eq:kerperp}
(\Ker(d))^{\perp} = \{f \in C^0(G, \RR) : \sum_{v\in V(G)}f(v)=0\} \ .
\end{equation}

There is a one-to-one correspondence between acyclic orientations of $G$ and the regions of $\H_G$ (see, e.g., \cite[Lem.~7.1 and Lem.~7.2]{MR712251}). In particular, the connected cuts  of $G$ are corresponding to 
the lowest dimensional regions of $\H_G$. Given any function $f \in C^0(G,\RR)$ one can label each vertex $v$ with the real number $f(v)$. In this way we obtain an acyclic partial orientation of $G$ by directing $v$ toward $u$ if $f(u) < f(v)$. Recall this means we have an acyclic orientation on the graph $G/f$ obtained by contracting all unoriented edges (i.e. all edges  $\{u,v\}$ with $f(u)=f(v)$).
We are mainly interested in acyclic orientations of $G$ with a {\em unique source} at $s \in V(G)$. For this purpose, we define
\vspace{-1mm}
\[
\H_{s}=\{f\in C^0(G, \RR): f(s)=-1 \}  \ .
\]
The restriction of the arrangement $\H_G$ to $\H_{s}$ will be denoted by $\H_G^{s}$. We denote the {\em bounded complex} (i.e. the polyhedral complex consisting of bounded cells) of $\H_G^{s}$ by $\B_G^{s}$. 
By \eqref{eq:kerperp}, the restriction of $\H_G$ to $\H_{s}$ coincides with the restriction of $\H_G$ to
\[
(\H_{s})'=\{f\in C^0(G, \RR): \sum_{v \ne q}f(v)=1 \}  \ .
\]
The regions of $\B_G^{s}$ are corresponding to acyclic orientations with a unique source at $s$ (see e.g., \cite[Theorem~7.3]{MR712251}). Fixing an orientation $\Oc$ of the graph 
$G$ will fix the linear forms $(d f)(e)=f(e_{+})-f(e_{-})$ for $e \in \Oc$ and gives an orientation to the hyperplane arrangement $\H_G^{s}$. The oriented matroid ideal associated to this oriented hyperplane arrangement $\H_G^{s}$ is denoted by 
$\C_G^s$ % (instead of $\OO_{\H_G^{s}}$)
\footnote{We use the notation $\C_G^s$ in order to be coherent with the notation used in \S\ref{sec:k-terminal}, however in \cite{novik} it 
is denoted by $\Oc_{\mathcal{M}}$, where $\mathcal{M}$ is the associated graphic matroid. We let $\C_G^s$, $\C_G$, $\Tf_G^s$ and $\Tf_G$, respectively, denote the ideals $\OO_{\H_G^{s}}$, $\bar{\OO}_{\H_G^{s}}$, $\OO_{\H_G^{s}}^\vee$ and $\bar{\OO}_{\H_G^{s}}^\vee$ from  \S\ref{subsec:oriented-matroid}.} and is called the {\em graphic oriented matroid ideal} associated to $(G,s)$. As an example we refer to \cite[Figure~11]{FatemehFarbod2} which is the bounded complex associated to the graph depicted in Figure~\ref{fig:spanning trees1}.

\vspace{-2mm}

\subsubsection{Toppling ideals}
Let $k[\zb]$ be the polynomial ring in the $n$ variables $\{z_v: v \in V(G)\}$. Any effective divisor $D=\sum_{v\in V(G)} D(v)$ with $D(v)\geq 0$ for all $v\in V(G)$, gives rise to a monomial
$
\zb^D= \prod_{v \in V(G)}{z_{v}^{D(v)}}.
$
Associated to every graph $G$ there is a canonical ideal 
\[
\begin{aligned}
I_G= \langle \zb^{D_1} - \zb^{D_2} : \, D_1 \sim D_2 \text{ both effective divisors}\rangle
\end{aligned}
\]
which encodes the linear equivalences of divisors on $G$ defined by Dhar \cite{Dhar90} (see also \cite{CoriRossinSalvy02}).
Once we fix a vertex $s$ of $G$ as a source, we denote $\MM_G^s$ for the initial ideal of $I_G$ with respect to 
the reverse lexicographic ordering %on $k[\xb]$
 induced by the total ordering on the variables compatible with the distances
 of vertices from $s$.% in $G$.
\begin{Remark}
The polyhedral cell complex $\B_G^{s}$ supports a minimal graded free resolution for  $\C_G^s$, $\C_G$ and $\MM_G^s$. The interested reader is referred to \cite{novik, FatemehFarbod2}  for more details.  
\end{Remark}

\section{Minimal free resolutions of system ideals }
\label{sec:k-terminal}
We now study the  ideals associated to the systems introduced in \S\ref{subsec:ST} and \S\ref{subsec:SMT} in details.

\subsection{Cut ideal $\C_{s,t}$}\label{sec:cut}
Here we precisely state the results in Theorem~\ref{intro:thm1}.
We recall that given a polyhedral complex and a subset $U$ of its vertices, its {\em induced subcomplex}  on $U$,
is the set of all its faces  whose vertices belong to $U$.  
We will show that the minimal free resolution of $\C_{s,t}$ is encoded in a subcomplex of the polyhedral subcomplex $\MB_G^{s}$ from \S\ref{sec:graphic}. %introduced in \S\ref{sec:graphic}.
%, supports the minimal free resolution of its corresponding ideal, i.e., the ideal generated by the monomials associated to its vertices, {\em %provided} that the subcomplex has been chosen {\em nicely}.
%\begin{Definition}\label{def:DG}

\medskip

Let $\Sc_{s,t}=\{\C_1,\ldots,\C_\ell\}$ be the set containing all connected cuts $\C_i=\EE(A_i,A_i^c)$ of $G$,
with $t\in A_i$ and $s\in A_i^c$. %, see \S\ref{sec:cutrel}. 
For each $i$, let $\cb_i$ denote the vertex of $\MB_G^{s}$ corresponding  to 
%those acyclic orientations of $2$-partitions such that $q$ and $t$ belongs to distinct partitions of the vertices. 
the cut  $\C_i$. 
We denote $\mathcal{D}_G^{s,t}$ for the induced subcomplex of $\MB_G^{s}$ on the vertices $\cb_1,\ldots,\cb_\ell$. 
By a slight abuse of notation, $\cb_i$ will be used to denote its corresponding $\mathbb{N}^{n}$-vector, where the $r$-th entry corresponding to the vertex $v_r$, is $|\{e\in \C_i:\ e_+=v_r\}|$.
%\end{Definition}
In order to have $\mathcal{D}_G^{s,t}$ as a polyhedral complex supporting the minimal free resolution of the ideal
\vspace{-1mm}
$$\C_{s,t}^\Oc=\langle \yb^{\C_i}=\prod_{e \in \C_i}{y_{e}}:\ \C_i\in \Sc_{s,t}\rangle,$$ 
we label the vertices of $\mathcal{D}_G^{s,t}$ by assigning the monomial  $\yb^{\C_i}$ as the label
of the vertex $\cb_i$.% for each $i$. 

\begin{Theorem} \label{thm:DG}
The labeled polyhedral cell complex $\mathcal{D}_G^{s,t}$ gives a $\mathbb{Z}^{2m}$-graded minimal free resolution for the ideals $\C_{s,t}^\Oc$ and $\C_{s,t}$. In particular, the $\beta_d(\C_{s,t}^\Oc)$ counts the $d$-dimensional bounded regions of $\mathcal{D}_G^{s,t}$ for all $d$.
\end{Theorem}

\begin{proof}
We first show that $\mathcal{D}_G^{s,t}$ supports the minimal free resolution of $\C_{s,t}^\Oc$. By \cite[Prop. 4.5]{MillerSturmfels} we only need to check that $(\mathcal{D}_G^{s,t})_{\leq\delta}$ is acyclic for any $\delta\in \mathbb{N}^n$, and the monomial labels of each pair of the cells $F_1\subsetneq F_2$ are different. The latter is clear because the labeling of the cells  are corresponding to the partitions of $G$, see e.g., \cite[p.112]{MR712251}.
We follow the strategy of the proof of Lemma 6.4 in \cite{Anton}. Assume that $\cb_1,\ldots,\cb_\ell$ are the vertices  of $(\mathcal{D}_G^{s,t})_{\leq\delta}$, and $\C_1,\ldots,\C_\ell$ are their corresponding cuts. 
Since the ideal $\C_{s,t}^\Oc$ and the labels of the cells are all squarefree, we can assume that $\delta$ is squarefree as well. Let $J_1=V(G)\backslash \{e_+: e\in\C_1\cup\cdots\cup\C_\ell\}$. Note that  $s\in J_1$, since $s$ is a source. Let $J_2\subset V(G)\backslash J_1$ be the connected component of $G[V(G)\backslash J_1]$ containing $s$. Now consider the subset $L=V(G)\backslash J_2$ and let $\pb$ be the $n$-vector $\frac{1}{|L|}\epsilon_L$, where $\epsilon_L$ is $1$ in entries corresponding to the vertices in $L$, and $0$ in other entries. We claim that $\pb$ is the star point of $(\mathcal{D}_G^{s,t})_{\leq\delta}$. Note that $\cb_1,\ldots,\cb_\ell$ are the vertices of $\MB_G^{s}$, and the graph obtained by contracting the vertices corresponding to $\pb$ will have a unique source. Assume that 
$\pb$ belongs to some cell  labeled by the monomial $a_\pb$ 
of $S$. Note that since $J_2\subset A_i^c$ for each $\C_i=\EE(A_i,A_i^c)$, we have $t\not\in J_2$. Thus $\pb$ belongs to $\mathcal{D}_G^{s,t}$. 

\smallskip

We now show that $a_\pb\subset \delta$. It is clear that $\supp(\cb_i)\subseteq L$ for all $i$, since one of the endpoints of each edge in the cut has positive indegree. Now consider $\ell\in L$ with $a_i>0$. 
Thus there exists an edge $e\in \C_1\cup\cdots\cup\C_\ell$, and a path from $s$ to $e_+$ not having intersection with any other vertex in $V(\C_1\cup\cdots\cup\C_\ell)$. Thus $e_-$ belongs to $V(\C_1\cup\cdots\cup\C_\ell)$ and so $\delta_i\geq 1$. 

\smallskip

Let $\rb$ be a point in $(\mathcal{D}_G^{s,t})_{\leq\delta}$. If $r_i>r_j$ for some edge $\{i,j\}$, then $i\in L$ and so $p_i=\frac{1}{|L|}\geq p_j$. Thus no hyperplane $\mathcal{H}_e$ (strictly) separates $\rb$ and $\pb$. Therefore the line segment $(\pb,\rb)$ connecting $\pb$ and $\rb$, sits inside some cell $R$ of $\MB_G^{s}$. Now we have to show that any interior point of $(\pb,\rb)$ is also in $(\mathcal{D}_G^{s,t})_{\leq\delta}$. %Note that 
The support of the monomial associated to $R$ is a subset of $\delta$ in $\MB_G^{s}$, and so in $\mathcal{D}_G^{s,t}$.
By \cite[Lem.~6.2]{FatemehFarbod2} the graph obtained by contracting the vertices corresponding to any interior point of $(\pb,\rb)$ has  a unique source, and so the region $R$ containing $(\pb,\rb)$ belongs to $\mathcal{D}_G^{s,t}$.

\smallskip
Note that the set $\{y_e-y_{\bar{e}}:\ e,\bar{e}\in\EE(G)\}$ forms a regular sequence for $S/\C_{s,t}^\Oc$.
Then by \cite[Thm.~A.11]{FatemehFarbod2} (see also \cite[Lem.~3.15]{EisenbudSyz}) relabelling the vertices of the complex $\mathcal{D}_G^{s,t}$ with monomials $\xb^{\C_i}$ (instead of $\yb^{\C_i}$ corresponding to cuts $\C_i$), gives us a polyhedral complex supporting the minimal free resolution of $\C_{s,t}$. %Note that as usual, 
We extend the labeling to all faces by the least common multiple rule.
\end{proof}

%%%%%%%%%%%%%%%%%%%%%%%%%%%%%%%%%%%%%%%%%%%%%%%%%%%%%%%%%%%%%%%%%%%%%%%%%%%%%%%%%%%%%%%%%%%%%%%%%%%%%%%%%%%%%%%
\vspace{-2mm}

\subsection{Spanning tree ideals  $\Tf_G$ and $\Tf_G^s$}\label{sec:span}
Here we present the SMT reliability formula as stated in Theorem~\ref{intro:thm2}.
%Algebraically, this corresponds to computing the multigraded Betti numbers of SMT ideals associated to graphs.  
However, the  results  are written in the algebraic language of ideals, the proofs are based on graph theoretical arguments. 
%and it will be of independent interest to scholars in combinatorial commutative algebra Our discussion of the statistical issues %centers around maximum likelihood estimation.
We derive the implicit representation of the higher syzygy modules of these ideals in terms of the acyclic orientations of graph
which also extends the results in \cite{sat}, where Satyanarayana and Prabhakar have presented a topological formula for the  SMT reliability formula by giving a combinatorial recipe to read the (non-cancelling) terms. % in the reliability formula.   
%%%%%%%%%%%%%%%%%%%%%%%%%%%%%%%%%%%%%%%%%%%%%%%%%%%%%%%%%%%%%%%%%%%%%%%%%%%%%%%%%%%%%%%%%%%%%%%%%%%%%%%%%%%%%%%
We remark that the ideals associated to spanning trees arise in different contexts, see e.g. \cite{novik, FatemehFarbod2,BerndMaria}.

\smallskip

The following special class  of acyclic partial orientations of $G$ arises naturally in our setting, where $g$ denotes the {\em genus} of the graph, i.e. $g=|E(G)|-|V(G)|+1$.
%\vspace{-2mm}
\begin{Definition}\label{def:span}
Fix a pointed graph $(G,s)$ with the (oriented) edge set $\EE(G)$. For each integer $0 \leq k \leq g$, an {\em oriented $k$-spanning tree}  \footnote{Gioan in \cite{Gioan} defines a $s$-connected orientation as an orientation %of $G$ 
in which every vertex is reachable from $s$ by a directed path. We use the notation $\Tf_k$ and the name oriented $k$-spanning tree 
in order to emphasize  that $\Tf_k$ is an acyclic subgraph of $G$ containing a rooted spanning tree with $k$ extra edges.
} $\Tc$ of $(G,s)$ is a connected subgraph of $G$ on $V(G)$  with a unique source at $s$ 
such that
\begin{itemize}
\item $\EE(\Tc)\subset \EE(G)$ with $|\EE(\Tc)|=n-1+k$,
\item $\Tc$ is acyclic.
\end{itemize}
The set of all oriented $k$-spanning trees of $(G,s)$ will be denoted by $\Sf_k(G, s)$.
The set $\Sf_0(G,s)$ corresponds to the set $\{\OO_T:\ T\ {\rm is\ a\ spanning\ tree\ of}\ G\}$.
%Note that our assumption that $\Tf$ has a unique source at $s$ implies that  for each  vertex $v\in V(G)\backslash \{s\}$, there exists an %oriented path from $s$ to $v$ in $\Tf$.
\end{Definition}

\vspace{-2mm}
\noindent
\begin{Notation} Let $\Tf$ be an element of $\Sf_{k}(G, s)$ for $k>0$, and $e\in\EE(\Tf)$. 
\begin{itemize}
\item[(1)] $\Tf^{e}$ denotes the  subgraph of $\Tf$ with the edge set $\EE(\Tf^e)=\EE(\Tf)\backslash\{e\}$. We set
$
W(\Tf)=\{e\in \EE(\Tf):\ \Tf^e\in\Sf_{k-1}(G, s)\}.
$

\item[(2)] Given an arbitrary ordering on the edges of the graph, say $\EE(G)=\{e_1,\ldots,e_m\}$, we set
$
{c(\Tf,e)=(-1)^j},
$
when $e$ is the $j$-th edge among the edges with endpoint $e_+$.
\end{itemize}
\end{Notation}

%%%%%%%%%%%%%%%%%%%%%%%%%%%%%%%%%%%%%%%%%%%%%%%%%%%%%%%%%%%%%%%%%%%%%%%%%%%%%%%%%%%%%%%%%%%%%%%%%%%%%%%%%%%%%%%

\begin{Theorem}\label{thm:GB}
Fix a pointed graph $(G,s)$. For each $k \geq 0$ there exists a natural injection
\[\psi_k:\Sf_{k}(G, s)\hookrightarrow \syz_{k}(\Tf_G^s)\]
such that
 the set $\Image(\psi_k)$ forms a minimal generating set for $\syz_{k}(\Tf_G^s)$.
\end{Theorem}

\begin{proof}
For $k=0$ the result is clear. Here $\psi_0(\Tf_G^s)$ is the generating set of $\Tf_G^s$, and
\[
\psi_{0}:\Sf_{0}(G, s)\hookrightarrow \syz_{0}(\Tf_G^s)=\Tf_G^s
\]
\[
\Tf \mapsto \prod_{e\in\EE(\Tf)} y_e\ .
\]
The proof is by induction on $k \geq 1$.
We define $\psi_k:\Sf_{k}(G, s)\hookrightarrow \syz_{k}(\Tf_G^s)$ by
\[
\Tf \mapsto \sum_{e\in W(\Tf)} (-1)^{c(\Tf,e)}y_e [\psi_{k-1}(\Tf^{e})]\ .
\]

{\bf Base case.}  Assume that $\Tf\in \Sf_{1}(G, s)$. Then  $\Tf$ has a unique cycle, and there exists a unique vertex $v$ such that indeg$(v)=2$. 
Therefore $W(\Tf)=\{e,e':\ e_+=e'_+=v\}$ and $c(\Tf,e)=-c(\Tf,e')$. Thus $\Tf^e, \Tf^{e'}$ belong to $\Tf_G^s$ and $y_e \psi_0(\Tf^e)- y_{e'} \psi_0(\Tf^{e'})=0$ which implies that
$
\psi_1(\Tf)=y_e [\Tf^e]- y_{e'}[\Tf^{e'}]
$
is an element of $\syz_1(\Tf_G^s)$.

Now assume that $\sum (-1)^{c(\Tf_i,e_i)} y_{e_i} [\Tf_i]$ is an element of the minimal generating set of $\syz_1(\Tf_G^s)$, that is,
$\sum (-1)^{c(\Tf_i,e_i)} y_{e_i}  \psi_0(\Tf_i)=0$. Note that here we use Remark~\ref{rem:CM}(ii) that the resolution of $\Tf_G^s$, as the Alexander dual of $\C_G^s$ is {\em linear}.
Therefore for each monomial $M_i=y_{e_i} \prod_{e\in \EE(\Tf_i)} y_{e}$ there exists a unique term $M_j=y_{e_j} \prod_{e\in \EE(\Tf_j)} y_{e}$ 
corresponding to an oriented tree $\Tf_j$ such that $M_i=M_j$. This implies that $\EE(\Tf_i)\cup \{e_i\}=\EE(\Tf_j)\cup \{e_j\}$. 
We claim that the induced graph $\Tf$ with the edges $\EE(\Tf)=\EE(\Tf_i)\cup \{e_i\}$ is acyclic.  Otherwise it should have a cycle $C$ including the edges $e_i$ and $e_j$. Let $v_i=e_{i_+}$ and $v_j=e_{j_+}$ with $v_i\neq v_j$ (the case $v_i=v_j$ is clear). On the other hand,
since $e_i\not\in \Tf_i$, there exists a (unique) path $P$ with $\EE(P)\subset \EE(\Tf_i)$ from $s$ to $v_i$ not going through the edge $e_i$. Similarly, there exists a (unique) path $P'$  with $\EE(P')\subset \EE(\Tf_j)$ from $s$ to $v_j$ 
not going through the edge $e_j$.  Thus the subgraph with the edge set $\EE(P)\cup \EE(P')\cup\EE(C)\backslash \{e_j\}\subseteq \EE(\Tf_j)$ contains a cycle which is a contradiction by our assumption that 
$\Tf_j\in \Sf_0(G)$ and it is acyclic.

\medskip

{\bf Induction hypothesis.} Now let $k>1$ and assume $\Image(\psi_{k-1})\subseteq \syz_{k-1}(\Tf_G^s)$
forms a minimal generating set for $\syz_{k-1}(\Tf_G^s)$.

%\medskip

Now assume that $\sb=\sum (-1)^{c(\Tf_i,e_i)} y_{e_i} [\Tf_i]$ is an element of the minimal generating set of $\syz_{k-1}(\Tf_G^s)$, that is,
$\sum (-1)^{c(\Tf_i,e_i)} y_{e_i}  \psi_{k-1}(\Tf_i)=0$ since by Remark~\ref{rem:CM}(ii) we know that the resolution of $\Tf_G^s$ is {\em linear}.
Note that corresponding to each term $M_{ij}=y_{e_i}  y_{e_j} \psi_{k-2}(\Tf_i^{e_j})$ with $e_j\in W(\Tf_i)$, there exists a unique term $y_{e_j}  \psi_{k-1}(\Tf_j)$ with $e_i\in W(\Tf_j)$ such that $M_{ji}=y_{e_j}  y_{e_i} \psi_{k-2}(\Tf_j^{e_i})$
is equal to $M_{ij}$. In particular, we have $\Tf_j\backslash\{e_i\}=\Tf_i\backslash\{e_j\}$. Set $\Tf$ be the  subgraph of $G$ with the edge set $\EE(\Tf_i)\cup\{e_i\}$.
Now we show that 

\begin{itemize}
\item[(1)] the subgraph $\Tf$ belongs to $\Sf_k(G,s)$. 
\item[(2)] all terms of $\psi_k(\Tf)$ have been appeared in $\sb=\sum (-1)^{c(\Tf_i,e_i)} y_{e_i} [\Tf_i]$.
\end{itemize}
%\medskip

(1): Set $v_i=e_{i_+}$ and $v_j=e_{j_+}$. If $v_i=v_j$ we are done. Assume that $v_i\neq v_j$. Assume by contradiction that $\Tf$ is not acyclic. 
Hence it contains an oriented cycle $C$ including the edges $e_i$ and $e_j$. 
Our assumption that $\Tf_i,\Tf_j\in \Sf_{k-1}(G,s)$ implies that 
there exists a path $P$ from $s$ to $v_i$ not going through the edge $e_i$, and a path $P'$ from $s$ to $v_j$ not going through the edge $e_j$. 

It is enough to show that there exists an edge $e_\ell\in W(\Tf_i)\backslash \EE(C)$. Since the only possible term cancelling 
 $M_{i,\ell}=y_{e_i}  y_{e_\ell} \psi_{k-2}(\Tf_i^{e_\ell})$ is the term $M_{\ell,i}$ coming from $\Tf_\ell$. However $\Tf_\ell$ contains $e_i$ and the oriented cycle $C$ which implies that $\Tf_\ell$ does not belong to $\Sf_{k-1}$. 
 Let $P'$ be the path in $\Tf_i$ from $s$ to $v_j$, and let $e'$ be the edge in $P'$ with $e'_+=v_j$. Then $e'\in W(\Tf_i)$ and  $e'\not\in\EE(C)$, as desired. 
%\smallskip

(2): Assume that $e\in W(\Tf)$. Then $\Tf^{e}$ is acyclic, and so $(\Tf^{e})^{e_i}$ is also acyclic. If $e_i\in W(\Tf^{e})$, then we are done. Assume that $e_i\not\in W(\Tf^{e})$. 
This implies that $e_+=e_{i_+}$ and there exists no other such edge. Note that in this case $e\in W(\Tf_j)$. 
Therefore in order to cancel the term corresponding to $(\Tf^{e_j})^{e}$ the term corresponding to $\Tf^e$ should be in $\sb$, which completes the proof.
\end{proof}

\begin{Remark}\label{rmk:inithm}
The resolution of the ideal $\Tf_G$ can be obtained from the resolution  of the ideal $\Tf_G^s$, by replacing all the variables $y_e$ with $x_e$, i.e., by forgetting the orientation on $e$, since
$\{y_e-y_{\bar{e}}:\ e,\bar{e}\in \EE(G)\}$ forms a regular sequence for %$\C_G^s$ and %its Alexander dual ideal 
$\Tf_G^s$
by Proposition~\ref{lem:sameBetti}(i).
\end{Remark}

\begin{Corollary}\label{cor:reg}
The projective dimension of %$S/\Tf_G^s$ and 
$\Tf_G$ and the Castelnuovo-Mumford regularity of  %$S/\C_G^s$ and 
$\C_G$  are equal to the genus of $G$. The minimal prime decomposition of $\Tf_G=\bigcap_{\C} P_\C$ is given in terms of the minimal cuts of graph, %$\Tf_G=\bigcap_{\C} P_\C$, where the intersection being over all connected cuts of $G$ and 
where $P_\C=\langle x_i:i\in E(\C) \rangle$.
\end{Corollary}
%%%%%%%%%%%%%%%%%%%%%%%%%%%%%%%%%%%%%%%%%%%%%%%%%%%%%%%%%%%%%%%%%%%%%%%%%%%%%%%%%%%%%%%%%%%%%%%%%%%%%%%%%%%%%%%
%%%%%%%%%%%%%%%%%%%%%%%%%%%%%%%%%%%%%%%%%%%%%%%%%%%%%%%%%%%%%%%%%%%%%%%%%%%%%%

\begin{Corollary}\label{cor:bet}
The Betti numbers of  $\Tf_G$ is independent of the characteristic of the  field.  For all $i \geq 0$,
$
\beta_{i,n-1+i}(\Tf_G)=\beta_{i}(\Tf_G)=|\Sf_{i}(G, s)|.
$
Moreover for each $\js\in \mathbb{Z}^n$, 
$\beta_{i,\js}(\Tf_G)=|\Sf_{i,\js}(G, s)|,$
where  $|\Sf_i(G, s)|$ is the number of $i$-spanning trees of $G$ and $|\Sf_{i,\js}(G, s)|$ is the number of  acyclic orientations on the induced subgraph on edges corresponding to $\js$.
\end{Corollary}

\begin{Example}\label{exam:cactus}
Let $G$ be a cactus, i.e., a connected graph in which each edge belongs to at most one cycle. Assume that the induced cycles of $G$ are $C_1,\ldots,C_k$ with $|V(C_i)|=n_i$ for $i=1,\ldots,k$. 
Then one can easily see, by induction on $k$, that% the number of spanning trees of $G$ is $|\Sf_{0,|E(G)|-k}(G,s)|=n_1n_2\cdots n_k$  and for $i>0$, $j=i+|V(G)|-1$ we have 
\[
|\Sf_{i,j}(G, s)|=\sum_{D=\{\ell_1,\ell_2,\ldots,\ell_i\}\atop
1\leq \ell_1<\ell_2<\cdots<\ell_i\leq k}  (n_{\ell_1}-1)(n_{\ell_2}-1)\cdots (n_{\ell_i}-1)(\prod_{a\in [k]\backslash D}\ n_a)\ .
\]
In particular if $n_i=n$ for all $i$, then $
|\Sf_{i,j}(G, s)|=\beta_{i,j}=
{k\choose i-1} (n-1)^{i-1} n^{k-i+1}.
$
\end{Example}

%%%%%%%%%%%%%%%%%%%%%%%%%%%%%%%%%%%%%%%%%%%%%%%%%%%%%%%%%%%%%%%%

%%%%%%%%%%%%%%%%%%%%%%%%%%%%%%%%%%%%%%%%%%%%%%%%%%%%%%%%%%%%%%%%%%%%%%%%%

\begin{Example}\label{exam:1}
For the graph $(G,s)$ depicted in Figure~\ref{fig:spanning trees1} with the fixed orientation $\Oc$,
\[
\Tf_G^s=\langle  
 y_{\bar{1}} y_{2}y_{\bar{3}},  y_{2} y_{\bar{3}} y_{5}, y_{1} y_{\bar{3}}y_{5},  y_{\bar{1}} y_{2}y_{4}, 
 y_{2} y_{4}y_{5},  y_{1} y_{4} y_{5}, y_{2} y_{3}y_{4},  y_{1} y_{3}y_{4}\rangle.
\]
Since $|\Sf_0(G,s)|=8$, $|\Sf_1(G,s)|=11$, $|\Sf_2(G,s)|=4$, we have
%the Betti numbers %of $\Rb/\Tf_G$ (and similarly $\Sb/\Tf_G^q$) 
%are:
$
\beta_{0,3}=8,\ \beta_{1,4}=11,\ \beta_{2,5}=4.
$
%%%%%%%%%
%%%%%%%%%%%%%%%%%%%%%%%%%%%%%%%%%%%%%%%
%%%%%%%%%%%%%%%%%%%%%%%%%%%%%%%%%%%%%%%

%%%%%%%%%%%%%%%%%%%%%%%%%%%%%%%%%%%%%%%
%%%%%%%%%%%%%%%%%%%%%%%%%%%%%%%%%%%%%%%
%%%%%%%%%%%%%%%%%%%%%%%%%%%%%%%%%%%%%%%

%%%%%%%%%%%%%%%%%%%%%%%%%%%%%%%%%%%%%%%
%%%%%%%%%%%%%%%%%%%%%%%%%%%%%%%%%%%%%%%
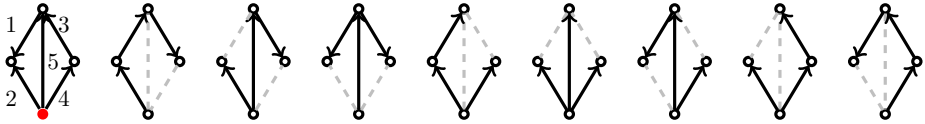
\begin{figure}[h!]

\begin{center}

\begin{tikzpicture} [scale = .14, very thick = 10mm]

    \node (n4) at (-6,1)  [Cred] {};  \node (m4) at (-5,6)  [Cwhite] {$5$};
  \node (n1) at (-6,11) [Cgray] {};   \node (m1) at (-4,9.5) [Cwhite] {$3$};
  \node (n2) at (-9,6)  [Cgray] {};  \node (m1) at (-9,9.5) [Cwhite] {$1$};
  \node (n3) at (-3,6)  [Cgray] {};  \node (m1) at (-4,2.5) [Cwhite] {$4$}; \node (m1) at (-9,2.5) [Cwhite] {$2$};
 \foreach \from/\to in {n4/n2,n3/n1,n1/n2,n4/n3,n4/n1}
  \draw[black][->] (\from) -- (\to);

  \node (n4) at (4,1)  [Cgray] {};
  \node (n1) at (4,11) [Cgray] {};
  \node (n2) at (1,6)  [Cgray] {};
  \node (n3) at (7,6)  [Cgray] {};

\foreach \from/\to in {n2/n1,n1/n3,n4/n2}
    \draw[][->] (\from) -- (\to);
 
\foreach \from/\to in {n4/n1,n3/n4}
    \draw[lightgray][dashed] (\from) -- (\to);

    \node (n4) at (14,1)  [Cgray] {};
  \node (n1) at (14,11) [Cgray] {};
  \node (n2) at (11,6)  [Cgray] {};
  \node (n3) at (17,6)  [Cgray] {};
  \foreach \from/\to in {n4/n2,n1/n3,n4/n1}
   \draw[][->] (\from) -- (\to);

     \foreach \from/\to in {n1/n2,n4/n3}
    \draw[lightgray][dashed] (\from) -- (\to);

    \node (n4) at (24,1)  [Cgray] {};
  \node (n1) at (24,11) [Cgray] {};
  \node (n2) at (21,6)  [Cgray] {};
  \node (n3) at (27,6)  [Cgray] {};
  \foreach \from/\to in {n4/n2,n4/n3}
    \draw[lightgray][dashed] (\from) -- (\to);

    \foreach \from/\to in {n4/n1,n1/n3,n1/n2}
    \draw[][->] (\from) -- (\to);

  \node (n4) at (34,1)  [Cgray] {};
  \node (n1) at (34,11) [Cgray] {};
  \node (n2) at (31,6)  [Cgray] {};
  \node (n3) at (37,6)  [Cgray] {};
   \foreach \from/\to in {n4/n1,n1/n3}
    \draw[lightgray][dashed] (\from) -- (\to);

    \foreach \from/\to in {n2/n1,n4/n2,n4/n3}
    \draw[][->] (\from) -- (\to);

    \node (n4) at (44,1)  [Cgray] {};
  \node (n1) at (44,11) [Cgray] {};
  \node (n2) at (41,6)  [Cgray] {};
  \node (n3) at (47,6)  [Cgray] {};
 \foreach \from/\to in {n1/n2,n3/n1}
    \draw[lightgray][dashed] (\from) -- (\to);

    \foreach \from/\to in {n4/n2,n4/n3,n4/n1}
    \draw[][->] (\from) -- (\to);

    \node (n4) at (54,1)  [Cgray] {};
  \node (n1) at (54,11) [Cgray] {};
  \node (n2) at (51,6)  [Cgray] {};
  \node (n3) at (57,6)  [Cgray] {};
 \foreach \from/\to in {n1/n3,n4/n2}
    \draw[lightgray][dashed] (\from) -- (\to);

    \foreach \from/\to in {n4/n3,n4/n1,n1/n2}
    \draw[][->] (\from) -- (\to);

%%%%%%%%%%%%

   \node (n4) at (64,1)  [Cgray] {};
  \node (n1) at (64,11) [Cgray] {};
  \node (n2) at (61,6)  [Cgray] {};
  \node (n3) at (67,6)  [Cgray] {};
 \foreach \from/\to in {n1/n2,n4/n1}
    \draw[lightgray][dashed] (\from) -- (\to);

    \foreach \from/\to in {n4/n3,n4/n2,n3/n1}
    \draw[][->] (\from) -- (\to);

%%%

 \node (n4) at (74,1)  [Cgray] {};
  \node (n1) at (74,11) [Cgray] {};
  \node (n2) at (71,6)  [Cgray] {};
  \node (n3) at (77,6)  [Cgray] {};
 \foreach \from/\to in {n4/n2,n1/n4}
    \draw[lightgray][dashed] (\from) -- (\to);

    \foreach \from/\to in {n4/n3,n3/n1,n1/n2}
    \draw[][->] (\from) -- (\to);

\end{tikzpicture}

\caption{Graph $G$ with a fixed orientation, and its oriented spanning trees}
\label{fig:spanning trees1}
\end{center}
\end{figure}

%%%%%%%%%%%%%%%
\end{Example}

\vspace{-7mm}

\subsection{Reduced divisors and oriented $k$-spanning trees}% and Dhar's burning algorithm}
\label{subsec:dhar}
%We study reduced divisors from an algebraic point of view. 
Using {\em Dhar's burning algorithm}, one can start from a source $s$, and check the availability of a vertex $v_1\in V(G)\backslash \{s\}$ with  $D(v_1)<|\EE({\{s\}},\{v_1\})|$. 
If there exists such a vertex, then the next step will check the same procedure to find a vertex $v_2$ with $D(v_2)<|\EE({\{s,v_1\}},\{v_2\})|$. The divisor is $s$-reduced if and only if by continuing this procedure we 
can enlarge the set containing $s$ to $V(G)$, (see \cite[Alg. 4]{FarbodMatt12}). We may assume that $D(s)=-1$  for all $s$-reduced divisors.

\medskip
The oriented $k$-spanning trees are naturally arisen in the theory of (reduced) divisors. 
%The following combinatorial result shows how the oriented $k$-spanning trees are naturally arisen in the theory of (reduced) divisors. 
%Later in \S\ref{sec:span} we study oriented $k$-spanning trees, and so the reduced divisors,
%from an algebraic point of view. 

\begin{Theorem}\label{thm:syz canonical divisor}
For each $\Tc\in \Sf_k(G, s)$, its corresponding divisor $D_\Tc$ is $s$-reduced. Associated to every $s$-reduced divisor $D$, there exists
 $\Tc\in \Sf_k(G, s)$ %an oriented $k$-spanning tree $\Tc$ 
with $D_\Tc\sim D$. 
\end{Theorem}

\begin{proof}
Let $\Tc$ be an oriented $k$-spanning tree. We show that $D_{\Tc}$ is $s$-reduced. %
Note that $V(G)\backslash \{s\}$ has at least one vertex, say $v_1$, with $\indeg_{V(G)\backslash \{s\}}(v_1)=0$, 
otherwise we have an oriented cycle in $G$. Thus  $\indeg_{\Tc}(v_1)=|\EE(\{v_1\},\{s\})|=|E(\{v_1\},\{s\})|$ and $|\EE(\{s,v_1\},  \{s,v_1\}^c)|=0$. Thus $\indeg(v_1)-1=D(v_1)<|E(\{s\},\{v_1\})|$. 
By continuing the same procedure we can enlarge $\{s,v_1\}$ to $V(G)$, as needed in Dhar's algorithm which implies that $D_{\Tc}$ is $s$-reduced. We also know that $D_\Tc(s)=-1$.

\smallskip

Let $D$ be a $s$-reduced divisor of degree $k-1$ (with $D(s)=-1$). First we find a partial orientation $\P$ with $D_\P=D$. We may assume $D(v)<\deg(v)$ for each $v$. 
%Assume that we have an orientation $\P$ with the right number of edges connected to $v_1,\ldots,v_{i-1}$, and %$\indeg_\P(v_i)<D(v_i)$. 
Let $\P$ be an arbitrary orientation with $D_\P<D$. Assume that $\indeg_\P(v)<D(v)$. If there exists an unoriented edge adjacent to $v$, then we orient this edge toward $v$ to increase the indegree of $v$, and get closer to the desired orientation. 
Otherwise 
$|\EE(\{v\}^c,\{v\})\cap \P|>0$. So by applying Lemma~\ref{lem:cut} we obtain a cut $\C=\EE(C,C^c)$ with $v\in C$. Now by inverting this cut, we increase the indegree of $v$ and the obtained partial orientation is closer to $D$. 
Note that by the same argument used in proof 
Lemma~\ref{lem:cut}, in case that the indegree of $v$ is greater than $D(v)$, we can use an unoriented edge in $\P$, and keep exchanging pair of edges as in Remark~\ref{rem:un/indirect} to obtain an unoriented edge adjacent to $v$. So that we can {\em unorient} 
this edge to decrease the indegree of $v$. By continuing the same argument, we keep decreasing the number $D(v)-D_\P(v)$ for each vertex $v$ and we get the desired orientation.

Now we want to show that $\P\in \Sf_k(G, s)$. Our assumption that $D(s)=-1$ implies that $s$ is a source. We may assume that $\P={\P_{\{s\}}}$ as in Definition~\ref{def:prec1}. 
The idea is to start form $s$, and add all other vertices of $G$ to $\{s\}$, step-by-step, by applying  Dhar's algorithm so that all oriented edges  are directed from the set containing $s$ to its complement. 
Since $D$ is $s$-reduced, there exists $v_1\in V(G)\backslash\{s\}$ such that no edge is directed to $v_1$ in $V(G)\backslash\{s\}$. Therefore, $|\EE(\{s,v_1\}^c, \{s,v_1\})|=0$. If  $\EE(\{s,v_1\},\{s,v_1\}^c)\subset \P$, then
$\EE(\{s,v_1\},\{s,v_1\}^c)$ is a cut, and we perform a cut-inverse to increase the number of vertices with the property that there exists a path from $s$ to them. Note that inverting cuts will never produce a cycle.
Otherwise, $\EE(\{s,v_1\},\{s,v_1\}^c)\backslash \P$ contains at least an edge $e$. Our assumption on $\P_{\{s\}}$ implies that none of the vertices of $V(G)\backslash \{s,v_1\}$ is oriented to $v_2=e_-$. 
Then we use the same argument for $\{s,v_1,v_2\}$, and by continuing the same procedure, we keep enlarging the set containing $s$ to $V(G)$. 
\end{proof}

\vspace{-2mm}

\begin{Remark}
Note that in the proof of Theorem~\ref{thm:syz canonical divisor} corresponding to each divisor $D$ with $0\leq\deg(D)<g$, we have found a partial orientation $\P$ with $D\leq D_{\P}$. We are mostly interested in having the {\em equivalency} $D\sim D_{\P}$. We have  shown that if $D$ is $s$-reduced, then $\P$ is indeed an oriented
$\deg(D)$-spanning tree with $D\sim D_{\P}$.
It is also clear that $\deg(D)=\deg(D_\P)$ implies the {\em equality} $D=D_\P$. In particular the equality holds for $\deg(D)=g-1$,  since $\deg(D_\P)\leq g-1$ (see  \cite[Thm.~4.7]{ABKS} for the same statement).
\end{Remark}
\begin{Corollary}\label{cor:syz-div}
 For a  $s$-reduced divisor $D$ of degree $k-1$ with $D(s)=-1$, there exists $\Tf\in \Sf_k(G, q)$ and a $k$-syzygy element $\psi_k(\Tf)$ of $\Tf_G$ such that $D=D_{\Tf}$.
\end{Corollary}

%\vspace{-2.5mm}

\begin{Remark}{\rm
The  reduced divisors played a prominent role in Baker-Norine's proof of Riemann-Roch theory for finite graphs. 
While this paper was being prepared, the preprint \cite{Spencer}
 was posted on the arXiv by Spencer Backman who applies similar results %similar to \S\ref{sec:reduced} 
to provide a new proof of the Riemann-Roch theorem. 
However our perspective 
 is mostly geometric
combinatorics and commutative algebra. %, see Theorem~\ref{thm:GB}.
There are several {\em other} bijections in the literature, see e.g.,  \cite{BensonTetali08}, between the maximum $G$-parking functions, 
the set of spanning trees with no broken circuit and particular acyclic orientations of $G$.
}
\end{Remark}

\subsection{Reliability of the dual systems}\label{sec:dual} 
For a system $S$, its dual is defined such that a path set of $S$ is a cut set of its dual. In our setting, the ideal $\C_G$ (respectively  $\C_{s,t} $) is the Alexander dual of the ideal $\Tf_G$ (respectively $\P_{s,t}$), see Proposition~\ref{prop:primary}. %\cite[Prop. 8.1]{FatemehFarbod2}.
Hence by Alexander inversion formula \cite[Thm.~5.14]{MillerSturmfels}
\begin{eqnarray*}\label{rem:dual}
\mathcal{R}_{\C_{s,t}}(\xb) =1-\mathcal{R}_{\P_{s,t}}(1-\xb)\quad\text{and}\quad \mathcal{R}_{\C_G}(\xb) =1-\mathcal{R}_{\Tf_G}(1-\xb)\ .
\end{eqnarray*}
This connection enables us to obtain many numerical and intrinsic information about a network by looking instead at its dual arising  in a different setting. 

\smallskip
\noindent
\textbf{Acknowledgments.}{
The author is very grateful to Bernd Sturmfels and Volkmar Welker for many helpful conversations, and she would like to thank Lionel Levine,  Dinh Le Van and  Raman Sanyal for their comments on the first draft. She also thanks   
Eduardo S{\'a}enz-de-Cabez{\'o}n and Henry Wynn for introducing her to  system reliability theory.
The author was supported by the Alexander von Humboldt Foundation.
}

\bibliographystyle{alpha}
\bibliography{Betti2015}

\end{document}